\documentclass{article}

\usepackage{graphicx,amscd,amsmath,amssymb,verbatim}
\usepackage[dvips]{hyperref}
\usepackage[TS1,OT1,T1]{fontenc}

\usepackage{multido}

\usepackage{graphicx}
\usepackage{psfrag}                      
\usepackage{amssymb,amsthm,amsmath,eufrak, amscd} 
\usepackage{fancyhdr}                 
\usepackage[all]{xy}                       
\usepackage[latin1]{inputenc}
\usepackage{dsfont}
\usepackage{upgreek}

\usepackage{verbatim}

\newtheorem{satz}{Satz}[section]

\newtheorem{theo}[satz]{Theorem}
\newtheorem{lem}[satz]{Lemma}
\newtheorem{prop}[satz]{Proposition}

\theoremstyle{definition}

\newtheorem{rem}[satz]{Remark}
\newtheorem{df}[satz]{Definition}
\makeatletter 
\@addtoreset{equation}{section}
\makeatother

\DeclareMathAlphabet\mathbfcal{OMS}{cmsy}{b}{n}
\def\R{\mathbb{R}}
\def\N{\mathbb{N}}
\newcommand{\Om}{\Omega}
\def\C{\mathcal{C}}
\def\d{\epsilon}
\def\p{\mathit{e}}
\def\L{\mathcal{L}}

\def\A{\mathcal{A}}
\def\D{\mathcal{D}}

\def\E{\mathcal{E}}
\def\F{\mathcal{F}}

\def\Meas{\mathcal{M}}

\newcommand{\la}{\lambda}
\newcommand{\na}{\nabla}
\newcommand{\al}{{\alpha}}

\newcommand{\eps}{\varepsilon}

\newcommand{\wto}{\rightharpoonup}
\newcommand{\li}{{\langle}}
\newcommand{\re}{{\rangle}}
\newcommand{\Li}{{\Big\langle}}
\renewcommand{\Re}{{\Big\rangle}}

\def\pa{\partial}
\def\esssup{{\rm ess\ sup}}
\renewcommand{\div}{{\rm div}}

\DeclareMathOperator{\dom}{dom}
\DeclareMathOperator{\epi}{epi}
\DeclareMathOperator{\inter}{int}

\newcommand{\cR}{{\mathbb R}}
\newcommand{\non}{\nonumber}

\begin{document}

\title{Measure-valued solutions for models of
ferroelectric material behavior}
\author{Nataliya Kraynyukova%
\thanks{Corresponding author: Nataliya Kraynyukova, Fachbereich Mathematik, 
Technische Universit\"at Darmstadt, 
Schlossgartenstrasse 7, 64289 Darmstadt, Germany, email: kraynyukova@mathematik.tu-darmstadt.de, 
Tel.: +49-6151-16-3287}\,
, Sergiy Nesenenko%
\thanks{Sergiy Nesenenko, Fachbereich Mathematik, Technische Universit\"at Darmstadt, 
Schlossgartenstrasse 7, 64289 Darmstadt, Germany, email: nesenenko@mathematik.tu-darmstadt.de, Tel.: +49-6151-16-2788}
}

\date{\today}
\maketitle
\begin{abstract}
In this work we study the solvability of the initial boundary value problems, which model a quasi-static nonlinear behavior of ferroelectric materials. Similar to the metal plasticity the energy functional of a ferroelectric material can be additively decomposed into reversible and remanent parts. The remanent part associated with the remanent state of the material is assumed to be a convex non-quadratic function $f$ of internal variables. 
In this work we introduce the notion of the measure-valued solutions for 
the  ferroelectric models and show their existence in the
rate-dependent case assuming the
 coercivity of the function $f$. Regularizing the energy functional by a quadratic positive
 definite term, which can be viewed as hardening,
  we show the existence of measure-valued solutions
for the rate-independent and rate-dependent problems avoiding
 the coercivity assumption on $f$. \end{abstract}




\section{Introduction and setting of the problem}
\label{form}
Due to the ability of ferroelectric materials to transform a mechanical action into an electrical impulse and vice versa they are being used in a broad range of modern engineering devices as actuators and sensors. Recent technological developments enabled the reduction of the production costs for ferroelectric ceramics and thereby increased the interest to use them in the novel implementations. Demand for the reliable mathematical models, which on the one hand are capable to describe a complicated nonlinear electromechanical behavior of ferroelectric devices in order to optimize their design and predict failure processes and on the other hand are simple enough for numerical implementations, caused a rapid progress in this field in the last years.
 A brief review of recent advances in modeling of ferroelectric material behavior can be found 
 in \cite{La04}. In the present work we study the solvability of the nonlinear initial boundary value
 problems associated with phenomenological constitutive models of ferroelectrics 
 \cite{HuFl01,Kam01, La02, LaMa02,RomSchr05,SchrRom05}. 
Similar to models in the metal plasticity the type of ferroelectric models considered here
 is formulated within a thermodynamic framework by using the standard
 material relevant description method of an energy function and a flow rule.
  In contrast to the micro-electromechanical models, which contain a large number of internal 
  variables standing for the distribution and the volume interaction of ferroelectric domains, 
  the main goal of the phenomenological models mentioned above is to improve the speed and 
  the robustness of numerical implementations by keeping the number of internal variables 
  as small as possible. The models presented in \cite{HuFl01,Kam01, La02, LaMa02,RomSchr05,SchrRom05} and studied in this work use as internal variables only the remanent strain and 
  the remanent polarization. 
\paragraph{Setting of the problem.}
 The model equations are formulated as follows. 
 Let $\Om\subset\R^3$ be an open bounded set with the $C^1$-boundary
$\pa\Om$  and $S^3$ denote the set of symmetric $(3\times3)$-matrices. Unknown are the displacement field $u(t,x)\in\R^3$, the Cauchy stress tensor
$\sigma(t,x)\in S^3$, the remanent strain tensor $r(t,x)\in S^3$, the
electric potential $\phi(t,x)\in\R$, the vector of electric
displacements $D(t,x)\in\R^3$ and the vector of remanent polarization $P(t,x)\in\R^3$
in a material point $x$ at time $t$.  The symbols 
$$\eps(u(t,x))=\frac{1}{2}(\nabla_x u(t,x)+(\nabla_x u(t,x))^T)\in S^3$$ and 
 $$E(\phi(t,x))=-\nabla_x\phi (t,x)\in\R^3$$ denote the linearized strain tensor and
  the electric field vector, respectively ($\eps$ and $E$ for short). The fundamental assumption of the models under consideration is that  the strain tensor $\eps$ and the vector of electric displacements $D$ can be additively decomposed into reversible and irreversible parts, i.e. $$\eps=(\eps-r)+r,\hspace{2ex}D=(D-P)+P.$$
 In this case $\eps-r$ and $D-P$ are reversible and $r$ and $P$ are irreversible parts of $\eps$ and $D$, respectively. 
For $(t,x)\in \Om_T:= (0,T)\times \Om$
 the unknown functions satisfy the following system of equations
  \begin{subequations}
\label{eq:no0}
\begin{alignat}{3}
-{\rm div}\ \sigma&=b,&&\ &&\label{eq:no1}\\
 {\rm div}\ D&=q,&&\ &&\label{eq:no2}\\
  \sigma&=\C&&(\eps - r) - \p^{T}&& E,\label{eq:no3}\\
 D&= \p&&(\eps - r) +\d &&E+ P,\label{eq:no4}
 \end{alignat}
 \vspace{-6mm}
 \begin{equation}
\qquad\qquad {r_t\choose P_t}\in \pa g\left({\sigma -f_r\choose E-f_P}  -\L {r\choose P}\right)\label{eq:no5}
 \end{equation}
 completed by the initial conditions
\begin{align}
r(0,x)=r^{0}(x),\ P(0,x)= P^{0}(x),\ x\in\Om \label{eq:no6}
\end{align}
and the homogeneous Dirichlet boundary conditions
\begin{align}
u(t,x)=0,\ \phi(t,x)=0,\ (t,x)\in[0,T)\times \pa\Om. \label{eq:no7}
\end{align} 
\end{subequations}
 The equations (\ref{eq:no1}) and (\ref{eq:no2}) are the
equilibrium equation and the Gauss equation in a quasi-static case, respectively.  
Here, the function $b(t,x)\in\R^3$ denotes a given body force and
$q(t,x)\in\R$ is a given density of free charge carriers. 
  The functions
  $g, f :S^3\times\R^3\to\R$ in (\ref{eq:no5}) denote constitutive functions, the form of which are usually
   determined by experiments. 
  Based on the thermodynamical considerations we give in the next two paragraphs the precise conditions, which
  $g$ and $f$ should satisfy, and discuss the equation (\ref{eq:no5}).
  The mapping $\L:S^3\times\R^3\to S^3\times\R^3$ in equation (\ref{eq:no5}) is linear symmetric and positive semi-definite and stands for the hardening effects. This mapping is not contained in the engineering models considered here and is introduced because of mathematical reasons, which are discussed in the last two paragraphs of the introduction. An overview of the previous results concerning the existence theory for the ferroelectric models and the structure of the present work can be found in the last paragraph of this section as well.
 \vspace{1.5ex}\\
  Due to the additive splitting of  the strain tensor $\eps$ and the vector of electric 
displacements $D$ into the reversible and 
irreversible parts, the constitutive relations (\ref{eq:no3}), (\ref{eq:no4})
can be equivalently rewritten as follows
 \begin{align}
\label{const_piez}
\begin{array}{r}
\sigma=\\
E=
\end{array} 
\begin{array}{l}
(\C+\p^{\rm T}\d^{-1}\p)(\eps-r)-  \p^{\rm T}\d^{-1}(D-P)\\
-\d^{-1}\p(\eps-r)+ \d^{-1}(D-P),
\end{array}
\end{align}
that implies that the reversible parts of $\eps$ and  $D$  satisfy 
the constitutive equations of linear piezoelectricity. Here the mappings $\C:S^3\to S^3,\ \d:\R^3\to \R^3,\  \p:S^3\to\R^3$ are 
material dependent elastic, dielectric  and piezoelectric tensors, respectively.
In the engineering literature 
 \cite{HuFl01,Kam01,La02, LaMa02,RomSchr05,SchrRom05} 
the entries of the constitutive tensors $\C,\d$ and $\p$ often depend on the internal variables $r$ 
and $P$. For example, in \cite{LaMa02} the tensor $\p$ has the following form
\begin{align}
\label{piez_ten}
\p_{kij}=\frac{|P|}{P_s}(\p_{33}n_kn_in_j+\p_{31}n_k\al_{ij}+\frac 12 \p_{15}(n_i\al_{jk}+n_j\al_{ik})),
\end{align} 
where $n=\frac{P}{|P|}$, $\al_{ij}=\delta_{ij}-n_in_j$, $\p_{ij}$ are constants and $P_s$ is the remanent polarization saturation constant. 
However, because of the difficulties arising in the mathematical treatment of the problem 
(\ref{eq:no0}),
in the present work we suppose that the tensors $\C,\d$ and $\p$ are independent of the 
internal variables $r$ and $P$.
Our approach to the derivation of 
the existence of the solutions for (\ref{eq:no0}) relies heavily on the $L^p$-existence theory for  
elliptic systems with $2<p<\infty$. In order to apply such a theory to our purposes
we have to require that the entries of the tensors  $\C,\ \d$ and $\p$ are continuous functions of
$x\in\overline{\Om}$. But since it is expected that the functions
$r,P$ belong only to $L^p(\Om)$ for some $2<p<\infty$ one can not guarantee that the mappings $\C,\ \d$ and $\p$ 
possess this regularity. 
Therefore,
 we suppose that the tensors $\C,\ \d$ and $\p$ are independent of $P$ and $r$
 and continuous functions of $x\in\overline{\Om}$.
 Additionally, according to the engineering models considered here we assume that the mappings $\C$, $\d$ and $\p$ are 
linear and bounded and that $\C$ and $\d$ are symmetric and positive definite 
uniformly with respect to $x\in\Om$.

The method presented in this work can be easily generalized to the case of nonhomogeneous Dirichlet, Neumann or mixed boundary conditions.

\paragraph{ Thermodynamical considerations and choices of the function $g$.}
A general form of the energy function corresponding to the models considered here
 can be derived by using the constitutive relations (\ref{const_piez}) and 
 the Clausius-Duhem inequality. Although the different types of thermodynamic potentials are used in the literature (for example, the Helmholtz free energy function in \cite{LaMa02,La04} 
 or the enthalpy function in \cite{RomSchr05}) it is typical in modeling of the nonlinear behavior of ferroelectric materials
  to derive  model equations by means of the Helmholtz free energy in the form
   $\Psi=\Psi(\eps,D,r,P)$. 
   The main requirement is that the function $\Psi$  satisfies the Clausius-Duhem inequality
\begin{align}
\label{sec_law}
0\leq \sigma\dot{\eps}+E\dot{D}-\dot{\Psi}=(\sigma-\Psi_\eps)\dot\eps+(E-\Psi_D)\dot D-\Psi_r\dot r-\Psi_P \dot P.
\end{align}
The arguments in the thermodynamics of irreversible processes yield 
 that the equations 
\begin{align}
\label{rev_part}
\sigma=\Psi_\eps(\eps,D,r,P)\ {\rm and}\ E=\Psi_D(\eps,D,r,P)
\end{align}
 hold. The Clausius-Duhem inequality can be then reduced to the following inequality
\begin{align}
\label{sec_law1}
0\leq -\Psi_r\dot r-\Psi_P \dot P.
\end{align}
Integrating the relations (\ref{rev_part}) and using (\ref{const_piez}) we conclude that the free energy function can be represented in the form
\begin{align}\label{Free_Energy}
\Psi(\eps,D,r,P)=\Psi_{\rm rev}+\tilde f(r,P),
\end{align}
where $\Psi_{\rm rev}=\frac 12\left((\C+\p^{\rm T}\d^{-1}\p)(\eps-r),\eps-r\right)-\left(\p^{\rm T}\d^{-1}(D-P),\eps-r\right)+\frac 12 \left(\d^{-1}(D-P),D-P\right)$ is the reversible part of the energy. The function $\tilde f$ corresponds to the remanent state of the material under consideration and
is given by
\[\tilde f(r,P)=f(r,P)+\frac{1}{2}|{\cal L}^{1/2}(r, P)^T|^{2}.\]
The authors of the engineering models \cite{HuFl01,La02, LaMa02,RomSchr05,SchrRom05} make different assumptions concerning the form of the function $f$.  Their choices are usually based on the experimental results. Several examples of the function $f$ are given below. The quadratic term with the linear positive semi-definite operator ${\cal L}$ is not contained in the models considered here. It can be regarded as a hardening term and the reason of its introduction is discussed in the next two paragraphs. 

Since the entries of the given tensors $\C,\p$ and $\d$ are assumed to be continuous functions and  independent of $r$ and $P$, using the expression for $\Psi$ we rewrite the Clausius-Duhem inequality (\ref{sec_law1}) as follows
\begin{align}
\label{sec_law2}
0\leq (\sigma-\tilde f_r)\dot r + (E-\tilde f_P) \dot P.
\end{align}
The second law of thermodynamics (\ref{sec_law2}) restricts the choice of the function $g$
 in the equation (\ref{eq:no5}). The inequality (\ref{sec_law2}) holds if $g$ is a proper
 convex function. 
 Additionally,  we suppose that the function $g$ is lower semi-continuous.  
 In most models in \cite{HuFl01,La02, LaMa02,RomSchr05,SchrRom05} 
 the function $g$ is chosen as an indicator function of some bounded, closed and 
 convex set $K\subset S^3\times\R^3$ with $0\in K$, namely,
\begin{align}
\label{rate_ind_g}
g=I_K=\begin{cases}0,& x\in K,\\ +\infty,& x\not\in K.\end{cases}
\end{align} 
 This choice of the function $g$ corresponds to a rate-independent process. 
 Rate-dependent effects such as time-dependent relaxation of ferroelectric polycrystals have been observed 
experimentally as well. To describe the rate-dependent behavior of a
ferroelectric material  in \cite{La02} the function $g$
is chosen in the form of a polynomial. In the rate-dependent case we  require that the function 
$g$ satisfies the following two-sided estimate with $c_1,c_3>0$ and $c_2,c_4\geq 0$
\begin{eqnarray}\label{grow_cond1}
c_1|v|^p-c_2\le g(v)\le c_3|v|^p+c_4,
\end{eqnarray}
which holds for
any $v\in S^3\times\R^3$. The condition (\ref{grow_cond1}) implies that
\begin{eqnarray}
\label{grow_cond1'}
g^*(v)\ge d_1|v|^{p^*}-d_2,\label{grow_cond2}
\end{eqnarray}
for $d_1>0$, $d_2\geq 0$ and any $v\in S^3\times\R^3$, where $g^*$ is 
the Legendre-Fenchel conjugate of $g$ 
(see Appendix~\ref{BasicsConAna} for basics on convex analysis). Throughout the whole
work we assume that the number $p$ satisfies
 $2\leq p<\infty$ with $p^{*}$ such that $1/p+1/p^*=1$. 

\paragraph{ Possible choices of the function $f$.} In most models in the engineering literature the remanent part of the energy $f:S^3\times\R^3\to\R$  is given by a convex function whose domain $\dom(f)$ is a convex (possibly unbounded) open subset of $S^3\times\R^3$.
In \cite{RomSchr05,SchrRom05,LaMa02} it is assumed  that the function $f=f(r,P)$ depends  only on $P$. In particular, in \cite{RomSchr05,SchrRom05} $f$ has the following form
\begin{align}
\label{schr_f}
f(r,P)&=f(P)=\begin{cases}
\frac{P_s}{2}\left((1+\frac{(P,a)}{P_s})\ln(1+\frac{(P,a)}{P_s})+(1-\frac{(P,a)}{P_s})\ln(1-\frac{(P,a)}{P_s})\right),& |(P,a)|<P_s,\\
+\infty,& |(P,a)|\geq P_s,
\end{cases}
\end{align}
where $a\in\R^3$ is a given direction with $\|a\|=1$ and $P_s$ is a saturation constant,
and in \cite{LaMa02} the function $f$ is of the form
\begin{align}
\label{lama_f}
f(r,P)&=f(P)=\begin{cases}
-P_s^2\left(\ln(1-\frac{|P|}{P_s})+\frac{|P|}{P_s}\right),& |P|<P_s,\\
+\infty,& |P|\geq P_s.
\end{cases}
\end{align}
In Theorem \ref{existMain_rate_dep} we suppose that if the function $f$ depends only on $P$, then
 it has to satisfy the following coercivity condition
\begin{align}
\label{grow_cond3''}
f(r,P)=f(P)\geq a_1|P|^2-a_2,\ a_1>0,\ a_2\geq 0.
\end{align} 
The coercivity condition (\ref{grow_cond3''}) is satisfied by the function $f$ given by (\ref{lama_f}), but not by the function $f$ in (\ref{schr_f}). Thus, the result of Theorem \ref{existMain_rate_dep} can not be applied to the function $f$ defined by (\ref{schr_f}).

In \cite{La02} and \cite{HuFl01} the function $f$ depends on both internal variables $r$ and $P$ and satisfies the following coercivity condition 
\begin{eqnarray}
\label{grow_cond3'}
f(z)\geq b_1|z|^p-b_2,\ b_1>0,\ b_2\geq 0,\ z=(r,P).
\end{eqnarray}

The present work is especially focused on the rate-dependent processes with the function $g$ satisfying the polynomial growth condition (\ref{grow_cond1}).  
 Under the condition
(\ref{grow_cond1}) we prove the existence of the measure-valued solutions in the sense of Definition~\ref{WeakSol} for the model (\ref{eq:no0}) without the regularizing term, i.e. with  $\L=0$ (see Theorem \ref{existMain_rate_dep}). However, in this case we have to assume that the function $f$ satisfies one of the coercivity conditions given above, i.e. either
 (\ref{grow_cond3''}) if $f$ depends only on $P$ or (\ref{grow_cond3'}) if $f$ depends on both variables $r$ and $P$.  If the linear mapping $\L$ in (\ref{eq:no5})
   is positive definite, then we are able to prove the existence of measure-valued 
   solutions for the problem (\ref{eq:no0}) in the rate-dependent case
    without assuming the coercivity of the function $f$ 
   (see Theorem \ref{existMain_rate_indep}) as well.  
   
    To prove the existence of measure-valued solutions  for the rate-independent case (Theorem \ref{existMain_rate_indep}) the linear symmetric and positive definite
     mapping $\L:S^3\times\R^3\to S^3\times\R^3$ in the equation (\ref{eq:no5}) is introduced.
  As we mentioned in the previous paragraphs this mapping is not contained in the engineering models. The additional quadratic hardening term $(\L z,z)$ with $z={r\choose P}$
  in the energy function $\Psi(\eps,D,r,P)$ regularizes the model 
  (\ref{eq:no0}) with $g$ given by (\ref{rate_ind_g}) and the existence of measure-valued solutions in the sense of 
  Definition~\ref{WeakSol} can be obtained. 
  The well-posedness of the problem (\ref{eq:no0}) in the rate-independent case 
  without the regularizing term $\L$ is an open problem at the moment.  
  
 \paragraph{Previous results and structure of the present work.}
 The first existence result for the nonlinear ferroelectric models in the rate-independent case is obtained in \cite{MiTi06} via the energetic approach.
However, in order to use the compactness argument the authors of \cite{MiTi06}  
 regularize the energy function by the additive quadratic gradient term of internal variables 
 $\li\L\na z,\na z\re$ with $z=(r,P)$ and positive definite $\L$, which is not present in the energy function (\ref{Free_Energy}). For such a modification of the model the authors of \cite{MiTi06} prove the existence of strong solutions. Hereby the tensors $\C,\p,\d$ are allowed to depend on the internal variables. 
 For the free energy regularized by the term $\li\L z, z\re$, which can be 
regarded as the hardening, we mention the following existence results \cite{KraAlb11,MiTe04}. In these works the tensors $\C,\p,\d$ are independent of the internal variables. The
derivations of these results require that the Nemytskii operator $F:L^2(\Om)\to\bar{\R}$, generated
by the function $f:S^3\times\R^3\to\bar{\R}$, is Frechet differentiable in $L^2(\Om)$. The
last requirement is satisfied if and only if the function $\na f$ is affine.
  In \cite{KraAlb11,MiTe04} the existence and uniqueness of the strong solution is shown in the rate-independent case
  under the assumption $F\in C^{2,{\rm Lip}}(L^2(\Om, S^3\times\R^3))$ and $F\in C^{3}(L^2(\Om, S^3\times\R^3))$, respectively. In the rate-dependent case
it is believed that there are no mathematical results concerning 
 the existence of solutions. In the present work we show the existence of measure-valued solutions of the rate-dependent problem (\ref{eq:no0}), when $\L=0$, i.e. the energy function in (\ref{Free_Energy}) does not contain regularizing terms. In Section~\ref{main_result} we introduce and motivate the notion of measure-valued solutions for the ferroelectric model formed by equations (\ref{eq:no0}) as well as formulate the main existence results
 in Theorems \ref{existMain_rate_dep} and \ref{existMain_rate_indep} 
 for the rate-dependent case  with $\L=0$ and for both rate-dependent and rate-independent cases with the positive definite mapping  $\L$, respectively. For the rate-dependent model with $\L=0$,  we assume that $f$ satisfies either the coercivity 
 condition (\ref{grow_cond3''}) or (\ref{grow_cond3'}). The proofs of these
 existence results are given in the subsequent sections.
  We note here also that the measure-valued solutions generalize naturally the notion
of the strong solution of the problem (\ref{eq:no0}) investigated 
previously in \cite{KraAlb11,MiTe04,MiTi06}. 
 
In Section~\ref{exist_ellipt_ferro} we show that for some given functions $r$ and $P$ the system of equations (\ref{eq:no1})-(\ref{eq:no4}), (\ref{eq:no7}) is an elliptic system of partial differential equations. Since the proof of the main existence results to the problem (\ref{eq:no0}) relies heavily on 
the existence theory for the equations of linear piezoelectricity,
we use $L^p$-existence theory for elliptic systems of partial differential equations and present the main properties of the solutions of the linear piezoelectricity model in full details  in Section~\ref{exist_ellipt_ferro}. 

In Section~\ref{red} we reduce the system (\ref{eq:no0}) to the evolution problem (\ref{ev_eq}), (\ref{ini_ev_eq}). In Section~\ref{time_discr} we use the Rothe time-discretization method to construct an approximating problem (\ref{CurlPr3Dis}), (\ref{dis_ini}) and show that it has a unique solution. In the following Sections~\ref{apriori}, \ref{Esistence} we show the convergence of the approximating sequence and prove the main existence results.

 
\section{Statement of main results} 
\label{main_result}
In this section we introduce the notion of the measure-valued solutions of the problem 
(\ref{eq:no0}) and then state  the main results of the work.
For completeness, we give the definition of the strong solutions of the problem 
(\ref{eq:no0}). 
\begin{df}[Strong solution]
\label{StrongSol}
 A function $(u,\phi,r, P)$ such that
\[(u,\phi)\in W^{1,p^*}(0,T; W^{1,p^*}_0(\Om, {\R}^3 \times{\R})), \ \
(r, P)\in W^{1,p^*}(0,T;L^{p^*}(\Om, {S}^3\times{\R}^3))\]
is called the {\it strong solution} of the initial boundary value
problem (\ref{eq:no0}), if for every $t\in [0,T]$
the function $(u(t),\phi(t))$ is the weak solution of the boundary value problem
 (\ref{eq:no1}) - (\ref{eq:no4}), (\ref{eq:no7}) with the given $r(t)$ and
$P(t)$ and the evolution problem (\ref{eq:no5}), (\ref{eq:no6}) is satisfied pointwise.
\end{df}
Next, we define the notion of the measure-valued solutions for the
initial boundary value problem (\ref{eq:no0}).
\begin{df}[Measure-valued solution]
\label{WeakSol}
 A function $(u,\phi,r, P, \tau)$ such that
\[u\in W^{1,{p^*}}(0,{T};W^{1,{p^*}}_0(\Om,{\R}^3)),\ \ \phi\in W^{1,p^*}(0,{T};W^{1,p^*}_0(\Om,{\R})),\]
\[z\equiv(r, P)\in 
W^{1,{p^*}}(0,{T};L^{p^*}(\Om,{S}^3\times{\R}^3)),\ \ \tau\in L^\infty_w(\Om_T,\Meas (S^3\times\R^3)) \]
with 
\[z(t,x)=\int_{S^3\times\R^3}\xi\  \tau_{t,x}(d\xi)\]
is called the {\it measure-valued solution} of the initial boundary value
problem (\ref{eq:no0}), if for every $t\in [0,T]$
the function $(u(t),\phi(t))$ is the weak solution of the boundary value problem
 (\ref{eq:no1}) - (\ref{eq:no4}), (\ref{eq:no7}) with the given $r(t)$ and
$P(t)$, 
the initial conditions (\ref{eq:no6}) are satisfied pointwise and
the following inequality
\begin{align}
\label{def_meas_sol}
&\int_{\Om_t} g^*\left({r_t\choose P_t}\right)dsdx+\int_{\Om_t} g\left({\sigma \choose E}  -\L {r\choose P} -\F\right)dsdx\nonumber\\&\leq
\int_\Om \int_0^t\left({r_t\choose P_t},{\sigma-\na_rf(r,P) \choose E-\na_P f(r,P)}  -\L {r\choose P}\right) dsdx
\end{align}
with $\F(t,x)=\int_{S^3\times\R^3}{\na_r f\choose\na_P f}(\xi)\  \tau_{t,x}(d\xi)$ holds for a.e. $t\in(0, T)$. 
\end{df}
\begin{rem}
We note that  the integrability of the function $\Phi(s,x)=\left({r_t\choose P_t},{\sigma-\na_rf(r,P) \choose E-\na_P f(r,P)}  -\L {r\choose P}\right)$  is not required in Definition~\ref{WeakSol}.
We require the existence of  the double integral $\int_\Om \int_0^t\Phi(s,x) dsdx$, only.
\end{rem}
\begin{rem}
\label{meas_is_str}
As it is discussed in Appendix \ref{BasicsConAna}, (\ref{conv_ineq}),  if we could show 
 $\F(t,x)=\na_{(r,P)}f(r(t,x),P(t,x))$ in (\ref{def_meas_sol}), i.e. $\tau=\delta_{(r,P)}$, then the measure-valued solution $(u,\phi,r,P,\tau)$ became the strong solution. The integrability of $\Phi(s,x)$ follows then automatically from the inequality (\ref{YoungFenchelIneq}).
\end{rem}
Next, we state the main results of this work.

\begin{theo}\label{existMain_rate_dep} 
Let the functions  
$b \in W^{1,p}(0,T; L^{p}(\Om, {\R}^3))$, $q\in W^{1,p}(0,T; L^{p}(\Om, {\R}))$
and $(r^{(0)}, P^{(0)})\in L^2(\Om,{S}^3\times{\R}^3 )$ be given. Suppose that $\L=0$.
 Assume that the function $g: S^3\times\R^3\to \bar{\R}$ is convex, l.s.c. and satisfies   
growth conditions (\ref{grow_cond1}), (\ref{grow_cond2}). Let the function $f:S^3\times\R^3\to\bar{\R}$ be convex and such that $f\in C^1(\dom{f})$ and satisfies either the coercivity condition (\ref{grow_cond3''}) if $f$ depends on $P\in \R^3$ only or (\ref{grow_cond3'}) if $f$ depends on both $r\in S^3$ and $P\in\R^3$.

Then there exists a measure-valued solution $(u,\phi,r, P, \tau)$
of the  problem (\ref{eq:no0}).  Additionally,
 if $f$ satisfies (\ref{grow_cond3'}), then $\left({r_t\choose P_t},{\sigma\choose E}\right)$ is integrable over $\Om_t$ for a.e. $t\in (0,T)$.
\end{theo} 
If $\L$ is positive definite the following result holds.
\begin{theo}\label{existMain_rate_indep} 
Let the functions  
$b \in W^{1,p}(0,T; L^{p}(\Om, {\R}^3))$, $q\in W^{1,p}(0,T; L^{p}(\Om, {\R}))$
and $(r^{(0)}, P^{(0)})\in L^2(\Om,{S}^3\times{\R}^3 )$ be given. Assume that $\L$ is positive definite and $g: S^3\times\R^3\to \bar{\R}$ is convex, l.s.c. function, which satisfies   
either the growth conditions (\ref{grow_cond1}), (\ref{grow_cond2}) or has
 the form (\ref{rate_ind_g}).
Let the function $f:S^3\times\R^3\to\bar{\R}$ be convex and satisfy $f\in C^1(\dom{f})$. 

Then there exists a measure-valued solution $(u,\phi,r, P, \tau)$
of the problem (\ref{eq:no0}). Moreover, the function 
  $\left({r_t\choose P_t},{\sigma\choose E}-\L{r\choose P}\right)$ is integrable over $\Om_t$ for a.e. $t\in (0,T)$.
\end{theo} 
At the end of this section we  present the conditions which guarantee that a measure-valued solution of the problem (\ref{eq:no0}) is the strong one. The next remark motivates the
introduction of the measure-valued solutions.
\begin{rem}
 We note that in order to guarantee that the measure-valued solution is strong one has to
 show that
   the inequality (\ref{def_meas_sol}) is satisfied with $\F(t,x)=\na_{(r,P)}f(r(t,x),P(t,x))$. 
Indeed, if the inequality
\begin{align}
\label{dep_cond}
\int_\Om\int_0^t g\left({\sigma-\na_r f (r,P)\choose E-\na_P f(r,P)}  -\L {r\choose P}\right)dxds\leq\int_\Om \int_0^t g\left({\sigma \choose E}  -\F -\L {r\choose P}\right)dxds,
\end{align} 
holds for a function $g$ satisfying (\ref{grow_cond1}), (\ref{grow_cond2}) for a.e. $t\in (0,T)$, then the equivalence (\ref{conv_ineq}) yields that the measure-valued solution $(u,\phi,r, P, \tau)$ is strong. For the case $\L=0$ considered in Theorem~\ref{existMain_rate_dep} we just set $\L=0$ in (\ref{dep_cond}). 
For the function $g$ given by (\ref{rate_ind_g}) we need the following condition. In Theorem~\ref{existMain_rate_indep} we have proved that ${\sigma\choose E}-\L{r\choose P}-\F\in K$, what
yields that $g\left({\sigma\choose E}-\L{r\choose P}-\F\right)=0$ for the function $g$ defined by (\ref{rate_ind_g}).  
Then the verification of the following condition
\begin{align}
\label{indep_cond}
{\sigma-\na_r f (r,P)\choose E-\na_P f(r,P)}  -\L {r\choose P}\in K,\hspace{2ex} a.e. \ (t,x)\in\Om_t,
\end{align}
implies that  the second term on the left side of the inequality (\ref{def_meas_sol}) is equal to zero, and therefore it ensures that the measure-valued solution $(u,\phi,r, P, \tau)$ is strong. 
\end{rem}


\section{Existence for linear piezoelectric models}
\label{exist_ellipt_ferro}
In section \ref{red} we reduce the system of equations (\ref{eq:no0}) to a single evolution 
equation for the vector-function $(r,P)$. This equation is just the equation (\ref{eq:no5}) with
 the functions $\sigma$ and $E$ expressed through the functions $r$ and $P$. 
 In this section we establish the relation between the functions $(\sigma,E)$ and $(r,P)$ using the equations (\ref{eq:no1})-(\ref{eq:no4}) with the homogeneous Dirichlet boundary
   conditions for the functions $u$ and $\phi$. Let us 
 suppose first that $\C,\d$ and $\p$ are measurable bounded functions of $x\in\Om$. 
 For simplicity we drop the time dependence of the given and the unknown function in this section.

We use notations from section \ref{form} and rewrite the system of equations (\ref{eq:no1})-(\ref{eq:no4})  as follows 
\begin{align}
\label{eq_rew0}
\div\ \left(\begin{array}{rcr}\C\ \eps(u)&+& \p^{\rm T}\na\phi\\ -\p\ \eps(u)&+&\d\ \na\phi\end{array}\right)=\left(\begin{array}{c}-b\\ -q\end{array}\right)+\div\ \left(\begin{array}{c}\C r\\ -\p\ r-P\end{array}\right).
\end{align}
Next, we introduce the following 
notations: $U\equiv (u,\phi)^T$, $z\equiv(r,P)^T$, $B\equiv(-b,-q)^T$, 
\begin{equation}
\A\equiv\left( 
\begin{array}{cc}
\C & \p^{\rm T}\\
-\p &\d 
\end{array}
\right),\ {\rm and}\ \ 
\E\equiv\left( 
\begin{array}{cc}
\C  & 0\\
-\p &-I 
\end{array}
\right).
\end{equation}
Now we use the symmetric properties of the tensor $\C$ and rewrite the system of equations (\ref{eq:no1})-(\ref{eq:no4}), (\ref{eq:no7}) as follows
\begin{align}
\label{eq_rew1}
D_h\A_{ij}^{hk}(x)D_kU^j(x)&=B_i(x)+D_h\E_{ij}^{hk}(x)z_k^j(x),\ x\in\Om,\\
U(x)&=0,\ x\in \pa\Om.\label{ini}
\end{align}

Since the entries of the mappings $\C$, $\d$ and $\p$ are bounded measurable functions we can suppose that the same holds for the entries $\A_{ij}^{hk}(x),\ x\in \Om$ of the mapping $\A:S^3\times\R^3\to S^3\times\R^3$. And since  $\C$ and $\d$ are symmetric and positive definite uniformly with respect to $x\in\Om$  and the terms containing $\p$ cancel each other in the expression $\A_{ij}^{hk}(x)\eta_j^k\eta_i^h$ with $\eta\in S^3\times\R^3$ we obtain that there exists a constant $c_0>0$ such that the following ellipticity condition 
\begin{align}
\label{ineq_ellip}
\A_{ij}^{hk}(x)\eta_j^k\eta_i^h\geq c_0\eta_j^k\eta_j^k,
\end{align}
holds for every $\eta\in S^3\times\R^3$ uniformly with respect to $x\in\Om$. 

Next, we show that the system (\ref{eq_rew1}) - (\ref{ini}) has a unique weak solution
 $U\in W^{1,p}_0(\Om,\R^3\times\R)$ for every given $z\in L^p(\Om, S^3\times\R^3)$ and 
 $B\in W^{-1,p}(\Om, \R^3\times\R)$ for $1<p<\infty$. For this purpose we use the existence 
 results for elliptic systems of partial differential equations. We make different assumptions 
 on the entries of $\A$ for  $p=2$ and $p\neq 2$. If $p=2$, we can apply the Lax-Milgram 
 result for the bilinear form $A(U,V)=\int_{\Om}(\A_{ij}^{hk}(x)D_jU^k(x),D_hV^i(x))dx$ to 
 the problem (\ref{eq_rew1}) - (\ref{ini}). In this case it is enough to suppose 
 that (\ref{ineq_ellip})  is satisfied a.e. $x\in\Om$ and that the entries of $\A$ are 
 measurable bounded functions. 

 We prove now that for every given $z\in L^2(\Om, S^3\times\R^3)$, $B\in W^{-1,2}(\Om,\R^3\times\R)$ there exists a unique weak solution $U\in W_0^{1,2}(\Om,\R^3\times\R)$, which means that $U$ satisfies
 \begin{subequations}
 \label{e2}
 \begin{align}
 \label{e21}
 A(U,V)=l(V)
 \end{align}
 for all $V\equiv(v,\psi)\in W^{1,2}_0(\Om,\R^3\times\R)$, where 
 \begin{align}
 \label{e22}
 A(U,V)=\li\A\  DU,DV\re=\li \A (\eps(u),\na\phi)^T,(\eps(v),\na\psi)^T\re, 
 \end{align}
 and
 \begin{align}
 \label{e23}
l(V)&\equiv\li B,V\re+\li\E z, DV \re=\li b,v\re+\li q,\psi\re-\li \E z,(\eps(v),\na\psi)^T\re.
 \end{align}
 \end{subequations}
The function $A:W^{1,2}_0(\Om,\R^3\times\R)\times W^{1,2}_0(\Om,\R^3\times\R)\to\R$ 
in (\ref{e22}) is a bilinear form.  Taking $U=V$ and using the ellipticity condition (\ref{ineq_ellip}) 
together with the inequalities of Korn and Poincare we obtain that $A(U,U)\geq c_1\|U\|^2_{1,2}$. 
Since the entries of $\A$ are bounded functions
we obtain that $|A(U,V)|\leq c_2\|U\|_{1,2}\|V\|_{1,2}$ and $|l(V)|\leq c_3\|V\|_{1,2}$. Therefore, 
the assumptions of the Lax-Milgram theorem  are satisfied and there exists a unique weak 
solution $U\in W_0^{1,2}(\Om,\R^3\times\R)$ of the problem (\ref{eq_rew1}) - (\ref{ini}). 
Thus, we have proved the following existence result for $p=2$:
\begin{theo}
\label{l_2_piez_exist}
Let $\Om\subset\R^3$ be an open bounded set with $\pa\Om\in C^1$, $(r,P)\in L^2(\Om, S^3\times\R^3)$ and $(b,q)\in W^{-1,2}(\Om, \R^3\times\R)$. Let the entries of the mappings  $\C:S^3\to S^3,\ \d:\R^3\to \R^3,\  \p:S^3\to\R^3$ be bounded measurable functions. We suppose that for a.e. $x\in{\Om}$ the mappings $\C$ and $\d$ are linear,  symmetric, they are positive definite uniformly with respect to a.e. $x\in\Om$, whereas $\p:S^3\to\R^3$ is just a linear mapping for a.e. $x\in\Om$. 

Then there exists a unique function $(u,\phi)\in W_0^{1,2}(\Om, \R^3\times\R)$, which satisfies the equations (\ref{eq:no1})-(\ref{eq:no4}) with  the homogeneous Dirichlet boundary conditions for arbitrary fixed $t\in [0,T)$ such that the estimate
\begin{align}
\label{ell_est_l2}
\|u\|_{1,2}+\|\phi\|_{1,2}\leq c(\|r\|_{2}+\|P\|_{2}+\|b\|_{-1,2}+\|q\|_{-1,2})
\end{align}
holds for some constant $c>0$, which is independent of $r,P,b$ and $q$.
\end{theo}
 
For the case $p\neq 2$ we can only prove the existence of the weak solution under the assumption that the functions $\A_{ij}^{hk}(x)$ are continuous for all $x\in\overline{\Om}$. Let $2<p<\infty$. We denote $f_i(x)\equiv B_i(x)+D_h\E_{ij}^{hk}(x)z_k^j(x)$. Suppose that $f\in W^{-1,p}(\Om, \R^3\times\R)$. For every $f\in W^{-1,p}(\Om, \R^3\times\R)$ one can find $F\in L^p(\Om,S^3\times\R^3)$ such that $F$ satisfies $f_i=D_h F^i_h$ in the sense of distributions and the estimate $\|F\|_p\leq c\|f\|_{-1,p}$ holds. We suppose that $\C,\d$ and $\p$ satisfy the assumptions of Theorem \ref{l_p_piez_exist}. Then since the assumptions of Theorem \ref{l_2_piez_exist} are also satisfied, there exists a unique weak solution $U\in W^{1,2}_0(\Om,\R^3\times\R)$ of the problem (\ref{eq_rew1}), (\ref{ini}). It is easy to prove that if $\C,\d$ and $\p$ satisfy assumptions of Theorem \ref{l_p_piez_exist}, then the functions  $\A_{ij}^{hk}(x)$ are continuous for all $x\in\overline{\Om}$ and satisfy the Legendre-Hadamard condition
\begin{align}
\label{leg_had}
\A_{ij}^{hk}(x)\eta_i\eta_j\xi_h\xi_k\geq c_0|\xi|^2|\eta|^2
\end{align}
for some $c_0>0$ and for every $x\in\overline{\Om}$, $\xi\in\R^4$ and $\eta\in\R^3$ 
uniformly with respect to $x$. It is shown in \cite[p. 373]{Giu03} that the function $U$ 
belongs then to $W_0^{1,p}(\Om,\R^3\times\R)$ and the estimate 
\begin{align}
\label{est_regul}
\|U\|_{1,p,\Om}\leq c\|F\|_{p,\Om}
\end{align}
holds with some $c>0$ independent of $F$. If we suppose that $z\in L^p(\Om, S^3\times\R^3)$, $B\in W^{-1,p}(\Om,\R^3\times\R)$, then we obtain that $F\in L^p(\Om, S^3\times\R^3)$ and $\|F\|_{p,\Om}\leq c(\|z\|_{p,\Om}+\|B\|_{-1,p,\Om})$. 

To prove that the conclusion of Theorem \ref{l_2_piez_exist} holds for $1<p<2$ as well
 we use the
 following duality arguments. Let $2<p<\infty$ and $p^*:\ \frac 1p+\frac 1{p^*}=1$. In the same way
 as above we prove that for any function $f\in W^{-1,p}(\Om, \R^3\times\R)$ there is a unique
 solution $U\in W^{1,p}_0(\Om)$ of the problem $D_h\A_{ji}^{kh}D_kU^j=f_i$ with 
the operator $\A$ replaced by $\A^T$ such that the inequality $\|U\|_{1,p,\Om}\leq c\|f\|_{-1,p,\Om}$ holds. Therefore we can define the linear bounded operator $T: W^{-1,p}(\Om)\to W^{1,p}_0(\Om)$ by $Tf=U$. Then there exists a unique operator $T^*:W^{-1,p^*}(\Om)\to W^{1,p^*}_0(\Om)$, such that
\begin{align}
[Tx,y]=[x,T^*y]
\end{align} 
holds for all $x\in  W^{-1,p}(\Om)$ and $y\in  W^{-1,p^*}(\Om)$. This proves that for every $g\in W^{-1,p^*}(\Om)$ the function $V\in W^{1,p^*}_0(\Om)$ defined by  $V\equiv T^*g$ satisfies   $D_h\A_{ij}^{hk}D_kV^j=g_i$ in the weak sense. The uniqueness follows immediately. Indeed, 
we take arbitrary $g\in W^{-1,p^*}(\Om)$ and $U\in W^{1,p}_0(\Om)$. The function $f_i\equiv D_h\A_{ji}^{kh}D_kU^j$ belongs to $W^{-1,p}(\Om)$ and since $U$ is the unique weak solution of $D_h\A_{ji}^{kh}D_kU^j=f_i$ we have also $Tf=U$. We obtain
\begin{align}
[U,g]=[Tf,g]=[f,T^*g]=[f,V]=-\int_{\Om}\left( D_h U^i(x),\A_{ij}^{hk}(x)D_kV^j(x)\right)dx
\end{align}
 for every $U\in W^{1,p}_0(\Om)$. Since $U$ is chosen arbitrary we get that $V\in W^{1,p^*}_0(\Om)$ satisfies $D_h\A_{ij}^{hk}D_kV^j=g_i$ in the weak sense and it is unique. From the relation $T^*g=V$ we obtain the estimate $\|V\|_{1,p^*,\Om}\leq c\|g\|_{-1,p^*,\Om}$. And since for every $h\in L^{p^*}$ the estimate $\|\div \  h\|_{-1,p^*}\leq \|h\|_{p^*}$ is satisfied we obtain that the following existence result holds for all $1<p<\infty$:
\begin{theo}
\label{l_p_piez_exist}
Let $\Om\subset\R^3$ be an open bounded set with $\pa\Om\in C^1$, $(r,P)\in L^p(\Om, S^3\times\R^3)$ and $(b,q)\in W^{-1,p}(\Om, \R^3\times\R)$ with $1<p<\infty$. Let the entries of the mappings  $\C:S^3\to S^3,\ \d:\R^3\to \R^3,\  \p:S^3\to\R^3$ be continuous functions of $x\in\overline{\Om}$. We suppose that for every $x\in\overline{\Om}$ the mappings $\C$ and $\d$ are linear, symmetric, they are positive definite uniformly with respect to $x\in\overline\Om$, and the mapping $\p:S^3\to\R^3$ is linear. 

 Then there exists a unique weak solution $(u,\phi)\in W^{1,p}(\Om, \R^3\times\R)$ of the problem (\ref{eq:no1})-(\ref{eq:no4}) with  the homogeneous Dirichlet boundary conditions for arbitrary fixed $t\in [0,\infty)$ and the estimate
\begin{align}
\label{est_lp}
\|u\|_{1,p}+\|\phi\|_{1,p}\leq c(\|r\|_{p}+\|P\|_{p}+\|b\|_{-1,p}+\|q\|_{-1,p})
\end{align}
holds with a constant $c>0$, which is independent of $r,P,b$ and $q$.
\end{theo}


\section{Reduction to the evolution equation}
\label{red}
In this section we show that the function $(\sigma,E)$ can be expressed conveniently
 through the function $z=(r,P)$ in such a way that after substituting $(\sigma,E)$ into the equations  (\ref{eq:no5}) and 
(\ref{eq:no6}) the problem (\ref{eq:no0}) is reduced to an evolution problem for the function $z$.

Let us suppose that the function $(r,P)$ is known and belongs to 
$L^p(\Om,S^3\times\R^3),\ 1<p<\infty$. We consider the equations (\ref{eq:no1})-(\ref{eq:no4}), (\ref{eq:no7}) and suppose that the entries of the mappings  $\C:S^3\to S^3,\ \d:\R^3\to \R^3,\  \p:S^3\to\R^3$ are continuous functions of $x\in\overline{\Om}$. 
Since the assumptions of Theorem \ref{l_p_piez_exist} are satisfied we get that for every given $z=(r,P)\in L^p(\Om,S^3\times\R^3)$ this problem has a unique solution $U=(u,\phi)\in W^{1,p}_0(\Om,S^3\times \R^3)$. 
Let us decompose $U=U_0+U_B$, where $U_0=(u_0,\phi_0)$ is a solution of 
the problem (\ref{eq:no1})-(\ref{eq:no4}) with $(b,q)=0$ and $(r,P)\neq 0$, and 
$U_B=(u_B,\phi_B)$ satisfies (\ref{eq:no1})-(\ref{eq:no4}) with $(b,q)\neq 0$ and $(r,P)= 0$. 

It follows from Theorem \ref{l_p_piez_exist} that for all $1<p<\infty$ the following estimate holds for the functions $u_0$ and $\phi_0$
\begin{align}
\label{est0}
\|u_0\|_{1,p}+\|\phi_0\|_{1,p}\leq  c(\|r\|_p+\|P\|_p).
\end{align}
Next, we define a linear operator $Q:L^p(\Om, S^3\times\R^3)\to L^p(\Om, S^3\times\R^3)$ by
\begin{align}
\label{proj_oper}
Q(r,P)^T\equiv (\eps(u_0), D_0)^T,
\end{align}
which is  bounded due to (\ref{est0}). It turns out that $Q$ is a projection operator. To this end,
we consider the functions $\tilde{\eps}_0=\eps_0-r$ and 
$\tilde{D}_0=D_0-P$ and rewrite the equations (\ref{eq:no3}), (\ref{eq:no4})
as follows
\begin{align}
\label{const_eq}
(\sigma_0,E_0)^T=\D(\tilde{\eps}_0,\tilde{D}_0)^T,
\end{align}
where $\D:S^3\times\R^3\to S^3\times\R^3$  is the operator defined by
\begin{equation}
\D=\left( 
\begin{array}{cc}
\C+\p^{\rm T}\d^{-1}\p & -\p^{\rm T}\d^{-1}\\
-\d^{-1}\p &\d^{-1} 
\end{array}
\right).
\end{equation}
It was shown in \cite{KraAlb11} that $\D$ is symmetric and positive definite.

\begin{lem} Let the vector $(r, P)^T\in L^p(\Om,S^3\times\R^3)$ be given. We define the linear mapping $Q=Q_{p}:L^p(\Om,S^3\times\R^3)\to L^p(\Om,S^3\times\R^3)$ by (\ref{proj_oper})
where $(\eps(u_0), D_0)^T$ satisfies the problem (\ref{eq:no1})-(\ref{eq:no4}) with $(b,q)=0$. Then $Q_p$ is the projection operator, which is adjoint to the operator  $Q_{p^*}$ with respect to the bilinear form $[z_1,z_2]_\D\equiv \langle \D z_1,z_2\rangle_{p,p^*}$.
\end{lem}
{\bf Proof:} 
The operator $Q_p$ maps the elements of the space $L^p(\Om,S^3\times\R^3)$ into the subspace
$\mathcal{H}=\{w=(\eps(u_0),D_0)^T:\ u_0\in W_0^{1,p}(\Om,\R^3),\ D_0\in L^p(\Om,\R^3):\ {\rm div}\ D_0=0\}$ of the space $L^p(\Om,S^3\times\R^3)$. Since (\ref{const_eq}) contains only the differences $\eps(u_0)-r,\ D_0-P$ we obtain from the uniqueness of the solution of the problem
(\ref{eq:no1})-(\ref{eq:no4}) with homogeneous Dirichlet boundary conditions that for every $w\in L^p(\Om,S^3\times\R^3)$ the operator $Q_p$ is the projection operator, namely, $Q_p^2w=Q_pw$.

Now we show that $Q_{p^*}$ is the adjoint operator to the operator $Q_p$ with respect to the bilinear form $[\cdot,\cdot]_\D$, i.e. we show that for every $z=(r,P)^T\in L^p(\Om,S^3\times\R^3)$ and $z^*=(r^*,P^*)^T\in L^{p^*}(\Om,S^3\times\R^3)$ the equality 
\begin{align}
\label{adj}
[Q_{p^*}z^*,z]_\D=[z^*,Q_pz]_\D
\end{align}
 holds. 

Let us denote the images of $Q_{p^*}z^*$ and $Q_pz$ as $Q_{p^*}z^*\equiv(\eps^*_0,D^*_0)^T$ and $Q_pz\equiv (\eps_0,D_0)^T$. We have
\begin{align*}
(\sigma_0,E_0)^T=\D((\eps_0,D_0)^T-(r,P)^T)=\D(Q_pz-z)
\end{align*}
and in the same way $(\sigma^*_0,E^*_0)^T=\D(Qz^*-z^*)$. We show now that $[(Q_{p^*}-I)z^*,Q_pz]_\D=0$ and $[(Q_{p}-I)z,Q_{p^*}z^*]_\D=0$ are satisfied, it would imply (\ref{adj}). We obtain
\begin{align*}
[(Q_{p^*}-I)z^*,Q_pz]_\D&=\li \D(Q_{p^*}z^*-z^*),Q_pz\re=\li (\sigma^*_0,E^*_0)^T,(\eps_0,D_0)^T\re=\li \sigma^*_0,\eps(u_0)\re\\&+\li -\na\phi^*_0,D_0\re=-\li \div\ \sigma^*_0,u_0\re+\li \phi^*_0,\div D_0\re=0.
\end{align*}
In the same way one can show that $[(Q_{p}-I)z,Q_{p^*}z^*]_\D=0$ holds. Then we have proved that $Q_p$ is the projection operator, which is adjoint to the operator  $Q_{p^*}$ with respect to the bilinear form $[\cdot,\cdot]_\D$.
$\Box$

Now, let us define $M:=\D(I-Q_p)$, $\hat{z}:=(\sigma_B,E_B)^T$, $z:=(r,P)^T$ and 
$z^0:=(r^0,P^0)^T$.
Inserting the expression $(\sigma_0,E_0)^T=\D(Q_p-I)(r,P)^T$ into the equation (\ref{eq:no5}) with the initial conditions (\ref{eq:no6}) yields  that equations (\ref{eq:no5}) - (\ref{eq:no6}) can be rewritten in the following abstract form 
\begin{align}
\label{ev_eq}
{z_t}&\in \pa I_g \left(-Mz-\L z-
\pa I_f (z)+\hat{z}\right),\\ 
z(0)&=z^0, \label{ini_ev_eq}
\end{align}
where $I_g,I_f:L^2(\Om,S^3\times\R^3)\to \R$ are functionals defined by (\ref{conv_int}).


\section{Existence and uniqueness for a time-discretized problem}
\label{time_discr}

We show the existence
of measure-valued solutions using the Rothe method (a time-discretization
method).  
In order to introduce a time-discretized problem, let us fix any
$m\in{\mathbb N}$ and set
$h:=\frac{T}{2^m}$. From the assumptions for the functions $b$ and $q$ we can conclude that $\hat{z}\in L^p(\Om,S^3\times\R^3)$. We set $$\hat{z}^n_m:=\frac{1}{h}\int^{nh}_{(n-1)h}\hat{z}(s)ds\in 
L^{p}( \Om, S^3\times\R^3),\ \ n=1,...,2^m.$$
Then we are looking for functions $z^n_m\in L^2(\Om,S^3\times \R^3)$ solving the following problem
\begin{align}
 \label{CurlPr3Dis} \frac{z^n_m-z^{n-1}_m}{h} & \in  
\partial I_g \big(\Sigma^n_m\big),\\ 
z_m^0&=z^0 \label{dis_ini}
\end{align}
with 
\begin{align}
\label{sigma}
\Sigma_m^n:=
-M_mz^n_m  
-\pa I_f (z^n_m)+\hat{z}_m^n\in L^2(\Om,S^3\times\R^3),
\end{align}

where 
$$M_m:=(\D(I-Q_2)+\L+\frac1mI): L^2(\Om,{S}^3\times\R^3)\to L^2(\Om,{S}^3\times\R^3).$$ 
To show that the discretized problem has a solution we need $M_m$ to be positive definite. This holds due to the term $\frac 1m I$ even if $\L$ is only positive semidefinite. Therefore we consider here  rate-dependent case with $\L=0$ and rate-independent case with $\L>0$ simultaneously and suppose that $\L$ is only positive semidefinite.
Recall that the functionals $I_f$ and $I_g$
are proper, convex and lower semi-continuous (see Section~\ref{BasicsConAna}).
We want to show that the equation
(\ref{CurlPr3Dis})  can be rewritten as
\begin{eqnarray}\label{ImportantEquiv}
\partial\Psi(z^n_m)\ni \hat z^n_m,
\end{eqnarray}
where
\[\Psi(v)=I_{g^*}\Big(\frac{v-p^{n-1}_m}{h}\Big)+\frac{1}{2}\|M^{1/2}_mv\|^{2}_2+I_f(v).\]
The functional $\Phi(v)=\frac{1}{2}\|M^{1/2}_mv\|^{2}_2:L^2(\Om,{S}^3\times\R^3)\to\bar\cR$
 is proper, convex and lower semi-continuous. Indeed, since $M_m$ is bounded and positive
definite operator, then it is maximal monotone by 
Theorem II.1.3 in \cite{Barb76}. Since the
operator $M_m$ is also self-adjoint, one has that $M_m=\partial\Phi$ by
Proposition II.2.7 in \cite{Barb76}. 
All other properties of $\Phi$ follow from its definition. The last thing
 which we
have to verify is whether the following relation
\[\partial\Psi=\partial I_{g^*}+\partial\Phi+\partial I_f\]
holds. By the definition of $\Phi$,
we conclude that the domain of $\Phi$
is equal to the whole space $L^2(\Om,{S}^3\times\R^3)$. By condition (\ref{grow_cond2})
the domain of the functional $I_{g*}$ is also the whole space $L^2(\Om,{S}^3\times\R^3)$ in the rate-dependent case. In the rate-independent case the domain of $I_{g*}$ also coincides with $L^2(\Om,{S}^3\times\R^3)$.  Therefore, condition (\ref{DomainSubdiffConvFunc}) is fulfilled and, since
all functionals are proper, convex and lower semi-continuous, 
Proposition~\ref{SumSubdiffConvFunc} gives the desired result.
With the relation (\ref{domainConvFunc})
 in hands the last observation implies that
\[\dom(\partial\Psi)=\dom(\partial I_f).\]
Since $\Phi$ is coercive in $L^2(\Om,{S}^3\times\R^3)$, 
which obviously yields the coercivity of $\Psi$, 
 the operator
$A=\partial\Psi$ is surjective by Theorem~\ref{SurjSubdifferential}.
Thus, we conclude that for every fixed $m\in\N$ and $n=1,...,2^m$ the problem (\ref{CurlPr3Dis}), (\ref{dis_ini})  has a solution $z_m^n\in L^2(\Om,S^3\times\R^3)$ for every given $\hat{z}_m^n\in L^p(\Om,S^3\times\R^3)$ and $z^0\in L^2(\Om,S^3\times\R^3)$. The solution $z_m^n$ is also unique. Indeed, suppose there are two functions $z_1$ and $z_2$, which satisfy the equation (\ref{ImportantEquiv}) for a given $\hat{z}_m^n$. We substitute the functions $z_1$ and $z_2$ into (\ref{ImportantEquiv}) and consider the difference of both equations. Then using the monotonicity of $\pa I_{g^*}$ and $\pa I_f$ we obtain that 
\begin{align*}
\li M_m(z_1-z_2), z_1-z_2\re\leq 0,
\end{align*}
which together with the positive definity of $M_m$ implies that the solutions coincide.
\vspace{1ex}\\
{\bf Rothe approximation functions:} 
For any family $\{\xi^{n}_m\}_{n=0,...,2^m}$ of functions in a reflexive Banach
space $X$, we define
{\it the piecewise affine interpolant} $\xi_m\in C([0,T],X)$ by
\begin{eqnarray}\label{RotheAffineinterpolant}
\xi_m(t):= \left(\frac{t}h-(n-1)\right)\xi^{n}_m+
\left(n-\frac{t}h\right)\xi^{n-1}_m \ \ {\rm for} \ (n-1)h\le t\le nh
\end{eqnarray}
and {\it the piecewise constant interpolant} $\bar\xi_m\in L^\infty(0,T;X)$ by
\begin{eqnarray}\label{RotheConstantinterpolant}
\bar\xi_m(t):=\xi^{n}_m\  {\rm for} \ (n-1)h< t\le nh, \ 
 n=1,...,2^m, \ {\rm and} \ \bar\xi_m(0):=\xi^{0}_m.
\end{eqnarray}
For the further analysis we recall the following property of $\bar\xi_m$
and $\xi_m$: 
\begin{eqnarray}\label{RotheEstim}
\|\xi_m\|_{L^s(0,T;X)}\le\|\bar\xi_m\|_{L^s(-h,T;X)}\le
\left(h \|\xi^{0}_m\|^s_X+\|\bar\xi_m\|^s_{L^s(0,T;X)}\right)^{1/s},
\end{eqnarray}
where $\bar\xi_m$ is formally extended to $t\le0$ by $\xi^{0}_m$ and
$1\le s\le\infty$ (see \cite{Roubi05}).\vspace{1ex}\\
\section{A-priori estimates.} 
\label{apriori}

{\bf Rate-dependent case.}
We suppose that $g$ satisfies the conditions (\ref{grow_cond1}) and (\ref{grow_cond2}).
Let us fix $m\in\N$ and $n=1,...,2^m$.  Since the problem (\ref{CurlPr3Dis}), (\ref{dis_ini}) has a unique solution, we obtain with the Young-Fenchel property (see Appendix \ref{BasicsConAna})
\begin{align*}
I_{g^*}(\frac{z_m^n-z_m^{n-1}}{h})+I_g(\Sigma_m^n)=\Li \frac{z_m^n-z_m^{n-1}}{h},\Sigma_m^n\Re,
\end{align*}
which together with the relation (\ref{sigma}) implies
\begin{align}
I_{g^*}&(\frac{z_m^n-z_m^{n-1}}{h})+I_g(\Sigma_m^n)+\frac 1h\Li z_m^n-z_m^{n-1}, M_mz_m^n\Re\nonumber\\&+\frac 1h\Li z_m^n-z_m^{n-1}, \pa I_f(z_m^n)\Re=\Li \frac{z_m^n-z_m^{n-1}}{h},\hat{z}_m^n\Re.\label{AprioriEstaHelp}
\end{align}

We note that $\li z^n_m-z^{n-1}_m, \pa \phi(z^n_m)\re\ge \phi(z^n_m)-\phi(z^{n-1}_m)$ holds
for any convex functional $\phi$. Therefore,
multiplying (\ref{AprioriEstaHelp}) by $h$ and summing the obtained 
relation for $n=1,...,l$ 
for any fixed $l\in\{1,...,2^m\}$ we derive the following inequality 
\begin{align}
& h\sum^l_{n=1}I_{g^*}\Big(\frac{z^n_m-z^{n-1}_m}h\Big)
+h\sum^l_{n=1}I_g(\Sigma_m^n)+\frac{1}2\Big(
\|(M+\L)^{1/2}z_m^l\|^2_2+\frac1m\| z^l_m\|^2_2\Big)+I_f(z^l_m)
\non\\&\le
\frac{1}2\Big(
\|(M+\L)^{1/2}z^0\|^2_2+\frac1m\| z^0\|^2_2\Big)+I_f(z^0)+h\sum^l_{n=1}\|\frac{z_m^n-z_m^{n-1}}{h}\|_{p^*}\|\hat{z}_m^n\|_p.\label{AprioriEstimHelp1}
\end{align}

 Applying the conditions (\ref{grow_cond1}) and (\ref{grow_cond2}) to the terms, which contain $I_g$ and $I_{g^*}$ and the Young inequality with $\eps<d_1$ to the last term in (\ref{AprioriEstimHelp1}) we obtain

\begin{align}
& h(d_1-\eps)\sum^l_{n=1}\|\frac{z_m^n-z_m^{n-1}}{h}\|_{p^*}^{p^*}
+hc_1\sum^l_{n=1}\|\Sigma_m^n\|_p^p+\frac{1}2\Big(
\|(M+\L)^{1/2}z_m^l\|^2_2+\frac1m\| z^l_m\|^2_2\Big)+I_f(z^l_m)
\non\\&\le
\frac{1}2\Big(
\|(M+\L)^{1/2}z^0\|^2_2+\frac1m\| z^0\|^2_2\Big)+I_f(z^0)+C_{\eps}h\sum^l_{n=1}\|\hat{z}_m^n\|_p^p+(d_2+c_2)|\Om_{T}|.\label{AprioriEstimHelp2}
\end{align}

 Now, taking Remark 8.15 in \cite{Roubi05} into account and using the definition
of Rothe's approximation functions we 
rewrite (\ref{AprioriEstimHelp2}) as follows
\begin{align}
& (d_1-\eps)\|\pa_t z_m\|_{p^*,\Om_T}^{p^*}
+c_1\|\bar{\Sigma}_m\|_{p,\Om_T}^{p}+\frac{1}2\Big(
\|(M+\L)^{1/2}\bar z_m(t)\|^2_2+\frac1m\| \bar z_m(t)\|^2_2\Big)\non\\&+I_f( z_m(t))
\le
\frac{1}2\Big(
\|M^{1/2}z^0\|^2_2+\frac1m\| z^0\|^2_2\Big)+I_f(z^0)+(d_2+c_2)|\Om_{T}|+C_{\eps}\|\bar{\hat{z}}_m\|_{p,\Om_T}^{p}.\label{AprioriEstimHelp3}
\end{align}
Since $\bar{\hat{z}}_m\to \hat{z}$ in $L^{p}(\Om_T)$, the last term in (\ref{AprioriEstimHelp3}) is bounded by a constant. The estimate (\ref{AprioriEstimHelp3}) implies that
\begin{align}
&\{z_m\}_m \ {\rm is}\  {\rm uniformly}\  {\rm bounded}\  {\rm in}
 \ W^{1,{p^*}}(0,{T};L^{p^*}(\Om,{S}^3\times{\R}^3)),
\label{aprioriEstim1}\\[1ex]
&\left\{\bar\Sigma_{m}\right\}_m\ {\rm is}\  {\rm uniformly}\  {\rm bounded}\  {\rm in}
 \ L^p(\Om_T,{S}^3\times{\R}^3),\label{aprioriEstim2}\\
&\{(M+\L)^{1/2} z_m\}_m \ {\rm is}\  {\rm uniformly}\  {\rm bounded}\  {\rm in}
 \ L^{\infty}(0,{T};L^{2}(\Om,{S}^3\times{\R}^3)),\label{aprioriEstim3}\\[1ex] 
& \left\{\frac1{\sqrt{m}} z_m\right\}_m \ {\rm is}\ 
  {\rm uniformly}\  {\rm bounded}\  {\rm in}
 \ L^{\infty}(0,{T};L^{2}(\Om,{S}^3\times{\R}^3)),\label{aprioriEstim4}\\[1ex]
&\{f(z_m)\}_m\ {\rm and}\ \{f(\bar{z}_m)\}_m\ {\rm are}\  {\rm uniformly}\  {\rm bounded}\  {\rm in}
 \ L^\infty(0,{T};L^{1}(\Om,{\R})). \label{aprioriEstim5}
\end{align}

We can improve the estimates (\ref{aprioriEstim2})-(\ref{aprioriEstim4}) and
show that $\left\{\bar\Sigma_{m}\right\}_m$ is uniformly bounded in the space $L^\infty(0,{T};L^{p}(\Om,{S}^3\times{\R}^3))$ and the sequences $\{(M+\L)^{1/2}\bar z_m\}_m$ and $\{\frac 1{\sqrt{m}}\bar z_m\}_m$ are uniformly bounded in $W^{1,2}(0,{T};L^{2}(\Om,{S}^3\times{\R}^3))$. Indeed, multiplying (\ref{CurlPr3Dis}) by the term $\frac{\Sigma_m^n-\Sigma_m^{n-1}}{h}$ and integrating over $\Om$ we obtain
\begin{align}
\label{estzmt01}
\Li \frac{z_m^n-z_m^{n-1}}{h}, \frac{\Sigma_m^n-\Sigma_m^{n-1}}{h}\Re=\frac 1h\Li \pa I_g(\Sigma_m^n),\Sigma_m^n-\Sigma_m^{n-1}\Re\geq \frac 1h (I_g(\Sigma_m^n)-I_g(\Sigma_m^{n-1})).
\end{align}
Then we use that $\pa I_f$ is a monotone operator and estimate the left side of (\ref{estzmt01}) from above
\begin{align}
\label{estzmt02}
&\Li \frac{z_m^n-z_m^{n-1}}{h}, \frac{\Sigma_m^n-\Sigma_m^{n-1}}{h}\Re=-\Li \frac{z_m^n-z_m^{n-1}}{h}, (M+\L+\frac 1m) \frac{z_m^n-z_m^{n-1}}{h}\Re\non\\ 
-&\Li \frac{z_m^n-z_m^{n-1}}{h}, \frac{\pa I_f(z_m^n)-\pa I_f(z_m^{n-1})}{h}\Re+\Li \frac{z_m^n-z_m^{n-1}}{h}, \frac{\hat{z}_m^n-\hat{z}_m^{n-1}}{h}\Re\non\\
\leq&-\|(M+\L)^{1/2} \frac{z_m^n-z_m^{n-1}}{h}\|^2_2-\|\frac 1{\sqrt{m}} \frac{z_m^n-z_m^{n-1}}{h}\|^2_2+\Li \frac{z_m^n-z_m^{n-1}}{h}, \frac{\hat{z}_m^n-\hat{z}_m^{n-1}}{h}\Re.
\end{align}
Now we combine (\ref{estzmt01}) and (\ref{estzmt02}), multiply  the obtained 
relation by $h$ and sum it up for $n=1,...,l$ 
and any fixed $l\in\{1,...,2^m\}$. We obtain 
\begin{align}
\label{estzmt02'}
h&\sum_{n=1}^l (\|(M+\L)^{1/2} \frac{z_m^n-z_m^{n-1}}{h}\|^2_2+\|\frac 1{\sqrt{m}} \frac{z_m^n-z_m^{n-1}}{h}\|^2_2)+\int_{\Om}g(\Sigma_m^l) dx\non\\
&\leq I_g(\Sigma(0))+h\sum_{n=1}^l \Li \frac{z_m^n-z_m^{n-1}}{h}, \frac{\hat{z}_m^n-\hat{z}_m^{n-1}}{h}\Re,
\end{align}
 which implies the estimate
 \begin{align}
 \label{estzmt03}
\|(M+\L)^{1/2} z_{mt}\|^2_{2,\Om_{T}}+\|\frac 1{\sqrt{m}} z_{mt}\|^2_{2,\Om_{T}}+\|\bar{\Sigma}_m(t)\|_{p,\Om}\leq I_g(\Sigma(0))+\|z_{mt}\|_{p^*,\Om_{T}} \|\hat{z}_{mt}\|_{p,\Om_{T}}. 
 \end{align}
 Since by (\ref{aprioriEstim1}) the rigth side of (\ref{estzmt03}) is bounded we obtain
\begin{align}
&\left\{\bar\Sigma_{m}\right\}_m\ {\rm is}\  {\rm uniformly}\  {\rm bounded}\  {\rm in}
 \ L^\infty(0,{T};L^{p}(\Om,{S}^3\times{\R}^3)),\label{aprioriEstim2'}\\
&\{(M+\L)^{1/2} z_m\}_m \ {\rm is}\  {\rm uniformly}\  {\rm bounded}\  {\rm in}
 \ W^{1,2}(0,{T};L^{2}(\Om,{S}^3\times{\R}^3)),\label{aprioriEstim3'}\\[1ex] 
& \left\{\frac1{\sqrt{m}} z_m\right\}_m \ {\rm is}\ 
  {\rm uniformly}\  {\rm bounded}\  {\rm in}
 \ W^{1,2}(0,{T};L^{2}(\Om,{S}^3\times{\R}^3)).\label{aprioriEstim4'}
 \end{align}
 
 In conclusion we note that if the function $f$ depends only on $P$ and satisfies the coercivity condition (\ref{grow_cond3''}),  we would obtain from (\ref{aprioriEstim5})
 \begin{align}
&\{P_m\}_m\ {\rm and}\ \{\bar{P}_m\}_m\ {\rm are}\  {\rm uniformly}\  {\rm bounded}\  {\rm in}
 \ L^\infty(0,{T};L^{2}(\Om,{S}^3\times{\R}^3))
\label{aprioriEstim6}
\end{align}
and if $f$ depends on both $r$ and $P$ and satisfies the coercivity condition (\ref{grow_cond3'}),  we would get then from (\ref{aprioriEstim5})
\begin{align}
&\{z_m\}_m\ {\rm and}\ \{\bar{z}_m\}_m\ {\rm are}\  {\rm uniformly}\  {\rm bounded}\  {\rm in}
 \ L^\infty(0,{T};L^{p}(\Om,{S}^3\times{\R}^3)),
\label{aprioriEstim7}. 
\end{align}
Here we emphasize that the estimates of this paragraph hold for positive 
semi-definite operator $\L$.

Suppose now that $\L$ is positive definite. Then the estimate (\ref{aprioriEstim3'}) immediately implies that
\begin{align}
\{z_m\}_m \ {\rm is}\  {\rm uniformly}\  {\rm bounded}\  {\rm in}
 \ W^{1,2}(0,{T};L^{2}(\Om,{S}^3\times{\R}^3))\label{aprioriEstim3''}
 \end{align}
 without coercivity assumptions for the function $f$.
 
{\bf Rate-independent case.} Now we suppose that $g$ is defined by (\ref{rate_ind_g}) and $\L$ is positive definite.
The proof runs the same lines of the second part of the previous paragraph except some 
slight changes. We multiply (\ref{CurlPr3Dis}) again by the term $\frac{\Sigma_m^n-\Sigma_m^{n-1}}{h}$, integrate over $\Om$  and use that $ I_g(\Sigma_m^n)=0,\ m,n\in\N$ to obtain
\begin{align}
\label{estzmt04}
\Li \frac{z_m^n-z_m^{n-1}}{h}, \frac{\Sigma_m^n-\Sigma_m^{n-1}}{h}\Re=\frac 1h\Li \pa I_g(\Sigma_m^n),\Sigma_m^n-\Sigma_m^{n-1}\Re\geq \frac 1h (I_g(\Sigma_m^n)-I_g(\Sigma_m^{n-1}))=0.
\end{align}
It follows from (\ref{estzmt04}) that
\begin{align}
\Li \left(M+\L+\frac 1m\right)\frac{z_m^n-z_m^{n-1}}{h}, \frac{z_m^n-z_m^{n-1}}{h}\Re&+\frac 1h\Li \pa I_f(z_m^n)-\pa I_f(z_m^{n-1}), z_m^n-z_m^{n-1}\Re\nonumber\\&\leq\Li \frac{z_m^n-z_m^{n-1}}{h}, \frac{\hat{z}_m^n-\hat{z}_m^{n-1}}{h}\Re.\label{estzmt05}
\end{align}
Since $M$ is positive semidefinite, $\pa I_f$ is monotone and $\L$ is positive definite we obtain the estimate
\begin{align*}
\left\|\frac{z_m^n-z_m^{n-1}}{h}\right\|_{2,\Om}&\leq 
C\left\|\frac{\hat{z}_m^n-\hat{z}_m^{n-1}}{h}\right\|_{2,\Om},
\end{align*}
which after multiplying by $h$ and summing for $n=1,...,l$ implies
\begin{align}
\label{n_est1}
\|z_{mt}\|_{2,\Om_{T}}&\leq C\|\hat{z}_{mt}\|_{2,\Om_{T}}.
\end{align}
Since the equation (\ref{AprioriEstaHelp}) holds also rate-independent $g$, we use it with the same arguments as in the previous paragraph to obtain the following estimate
\begin{align}
\label{n_est2}
&\int_0^t I_{g*}(z_{mt})dt+\int_0^t I_g(\bar{\Sigma}_m)dt+I_f(z_m(t))\nonumber\\&\leq \frac 12 (\|(M+\L)^{1/2}z^0\|_2^2+\frac1m\|z^0\|_2^2)+I_f(z^0)+\|z_{mt}\|_{2,\Om_{T}}\|\bar{\hat{z}}_m\|_{2,\Om_{T}}.
\end{align}
Since the sequence $\{\hat{z}_m\}_m$ converges strongly in $W^{1,p}(0,T;L^p(\Om))$ and the domain of $I_g$ is the bounded convex set $K$ we obtain in the rate independent case that the estimates
(\ref{n_est1}) and  (\ref{n_est2}) imply 
\begin{align}
&\{z_m\}_m \ {\rm is}\  {\rm uniformly}\  {\rm bounded}\  {\rm in}
 \ W^{1,{2}}(0,{T};L^{2}(\Om,{S}^3\times{\R}^3)),
\label{aprioriEst1}\\[1ex]
&\left\{\bar\Sigma_{m}\right\}_m\ {\rm is}\  {\rm uniformly}\  {\rm bounded}\  {\rm in}
 \ L^\infty(0,{T};L^{\infty}(\Om,{S}^3\times{\R}^3)),\label{aprioriEst2}\\
&\{f(z_m)\}_m\ {\rm is}\  {\rm uniformly}\  {\rm bounded}\  {\rm in}
 \ L^\infty(0,{T};L^{1}(\Om,{\R})). \label{aprioriEst5}
\end{align}

\section{Existence of measure-valued solutions}
\label{Esistence}

Based on the results of the previous sections we are able now to prove the main existence results
of this work, Theorem~\ref{existMain_rate_dep} and Theorem~\ref{existMain_rate_indep}.
\paragraph{The proof of Theorem~\ref{existMain_rate_dep}.} 
\begin{proof} 
In a similar way as in the beginning of section \ref{apriori} we use the equality (\ref{AprioriEstaHelp}) to derive the following inequality 
\begin{align}
\label{n_est3}
&\int_0^{t}\int_{\Om} g^*(z_{mt}(x,s))dsdx+\int_0^{t}\int_{\Om}g(\bar{\Sigma}_m(x,s))dsdx+ \li M \bar{z}_m,z_{mt}\re_{2,\Om_t}\non\\&+\int_{\Om}f(z_m(x,t))dx-\int_{\Om}f(z^0(x))dx 
\leq\frac1{2m}\| z^0\|^2_2+\li z_{mt},\bar{\hat{z}}_m\re_{2,\Om_t}.
\end{align}
 Using a-priori estimates from section \ref{apriori} we want to pass to the limit in the inequality (\ref{n_est3}).

First, we obtain by  
(\ref{aprioriEstim1}) that, at the expense of extracting a subsequence,  the sequence
$\{z_m\}_m$ converges weakly in the space  $ W^{1,p^*}(0,T;L^{p^*}(\Om))$ to some $z$. Next we claim that the sequence $\{\bar z_m\}_m$ converges weakly  in $L^{p^*}(\Om_T)$ and 
the weak limits of $\{\bar z_m\}_m$ and $\{z_m\}_m$ coincide. 
Indeed, using (\ref{aprioriEstim1}) this can be shown as follows
\begin{align}
&\|z_m-\bar z_m\|^{p^*}_{{p^*},\Om_{T}}=\sum_{n=1}^{2^m}
\int_{(n-1)h}^{nh}\left\|(z^{n}_m-z^{n-1}_m)\frac{t-nh}{h}\right\|^{p^*}_{p^*}dt\non\\
&=\frac{h^{{p^*}+1}}{{p^*}+1}\sum_{n=1}^{2^m}
\left\|\frac{z^{n}_m-z^{n-1}_m}{h}\right\|^{p^*}_{p^*}=\frac{h^{p^*}}{{p^*}+1}
\left\|\frac{dz_m}{dt}\right\|^{p^*}_{{p^*},\Om_{T}},\label{n_est4}
\end{align}
which implies that  $\bar z_m-z_m$ converges strongly to 0  in 
$ L^{p^*}(\Om_{T})$. Then  the sequence $\{\bar z_m\}_m$ converges weakly  in $L^{p^*}(\Om_T)$ to the same weak limit $z\in W^{1,p^*}(0,T;L^{p^*}(\Om))$ as $\{z_m\}_m$.  

In the same way,  using (\ref{aprioriEstim3'}) and (\ref{aprioriEstim4'}), we obtain that
\begin{align}
\label{n_est4'}
M(\bar{z}_m-z_m)\to 0\ {\rm and}\ \frac 1m(\bar{z}_m-z_m)\to 0\ {\rm in}\ L^2(\Om_T).
\end{align}

Since the elements of $\{\bar{\Sigma}_m\}_{m\in\N}$ belong to the space $L^p(\Om_{T})$, we can use the definition (\ref{conv_int}) to define the convex functionals $\hat{I}_g:L^p(\Om_t)\to\bar{\R}$ and $\hat{I}_{g^*}:L^{p^*}(\Om_t)\to\bar{\R}$ by $\hat{I}_{g^*}(z)\equiv\int_0^{t}\int_{\Om}g^*(z(s,x))dsdx$ and 
$\hat{I}_g(z)\equiv\int_0^{t}\int_{\Om}g(z(s,x))dsdx$, respectively. 
Then the functionals satisfy $\hat{I}_{g^*}=(\hat{I}_g)^*$. Since $z_{mt}$ converges weakly to
 $z_t$ in $L^{p^*}(\Om_T)$ we obtain 
\begin{align}
\label{n_est5}
\liminf_{m\to\infty}\hat{I}_{g^*}(z_{mt})\geq \hat{I}_{g^*}(z_{t}). 
\end{align}

Due to (\ref{aprioriEstim2}) there exists a subsequence of $\{\bar{\Sigma}_m\}_m$, which converges weakly to $\bar{\Sigma}$ in $L^p(\Om_T)$. Therefore we get
\begin{align*}
\liminf_{m\to\infty}\hat{I}_{g}(\bar{\Sigma}_m)\geq \hat{I}_{g}(\bar{\Sigma}). 
\end{align*}
Our next goal is to compute the function $\bar{\Sigma}$. To this end, we consider the sequence
 $\bar{\Sigma}_m=-M\bar{z}_m+\frac 1m \bar{z}_m-\pa f(\bar{z}_m)+\bar{\hat{z}}_m,\ m\in\N$. 

By constructions, $\bar{\hat{z}}_m$ converges strongly to $\hat{z}$ in $L^p(\Om_T)$. 
It follows from (\ref{aprioriEstim4'}) that the sequence $\{\frac{1}{\sqrt{m}}z_m\}_m$ is uniformly bounded in $L^2(\Om_T)$. By (\ref{n_est4'}) the sequence $\{\frac{1}{\sqrt{m}}\bar{z}_m\}_m$ is also uniformly bounded in $L^2(\Om_T)$ 
and therefore $\frac 1m\bar{z}_m $ converges strongly to $0$ in $L^2(\Om_T)$ . 

Moreover, it follows from (\ref{aprioriEstim3'}) and (\ref{n_est4'}) that the sequence $\{M\bar{z}_m\}_m$ converges weakly in $L^2(\Om_T)$ and the weak limit coincides with the weak limit of the sequence $\{M{z}_m\}_m=\{{\sigma_m\choose E_m}\}_m$. By Theorem \ref{l_p_piez_exist} 
we get that for every $z_m\in W^{1,2}(0,T;L^{2}(\Om))$ there exists a solution $(\sigma_m,E_m)\in W^{1,2}(0,T;L^2(\Om;S^3\times \R^3))$ of the problem 
\begin{align}
-{\rm div}\ \sigma_m&=0,\ \ \ {\rm div}\ D_m=0\label{equ1},\\
  {\sigma_m\choose E(\nabla \phi_m)}&=\D\left({\eps(\nabla u_m)\choose D_m}-z_m\right)=-\D(I-Q_2)z_m=-Mz_m\label{equ2}
  \end{align}
  with homogeneous boundary conditions for the functions $(u_m,\phi_m)$. 
  
  Now, let us  define the operator $M_{p^*}=\D(I-Q_{p^*}): L^{p^*}(\Om)\to L^{p^*}(\Om)$. 
 Due to the uniqueness of the solution of the problem (\ref{equ1}) - (\ref{equ2}),  instead 
 of the operator $M$ we can 
 consider the operator $M_{p^*}$  in (\ref{equ2}).
 We note that the operator $M_{p^*}$ is the extension of the operator $M$ on $L^2(\Om)$. Thus,
 because of the linearity of the operator $M_{p^*}$, the sequence $\{{\sigma_m\choose E_m}\}_m$ converges weakly in the space $W^{1,2}(0,T;L^2(\Om))$ to $M_{p^*}z={\sigma\choose E}$, which is the solution of the problem (\ref{equ1}), (\ref{equ2}) corresponding to the function $z$.   
Since the sequence $\{-M\bar{z}_m+\frac 1m \bar{z}_m+\bar{\hat{z}}_m\}_m$ converges weakly in $L^2(\Om_T)$ and $\{\bar{\Sigma}_m\}_m$
converges weakly in $L^p(\Om_T)$, the sequence $\{\pa I_f(\bar{z}_m)\}_m$ converges weakly in $L^2(\Om_T)$ to some $\F\in L^2(\Om_T)$. By Theorem \ref{young_exist} there exists a Young measure $\tau\in L_w(\Om_T; \Meas(S^3\times\R^3))$ associated with the sequence $\{z_m\}_m$ such that 
$\F(t,x)=\int_{S^3\times\R^3}\pa f(\xi)\tau_{t,x}(d\xi)\ {\rm for}\ {\rm a.e.}\ (t,x)\in \Om_T.$ Thus, we
get that $\bar{\Sigma}_m$ converges weakly to
 $-M_{p^*}z-\F+\hat{z}\ {\rm in}\ L^p(\Om_T)$ and that the inequality
holds
\begin{align}
\label{n_est6}
\liminf_{m\to\infty}\hat{I}_{g}(\bar{\Sigma}_m)\geq \hat{I}_{g}(-M_{p^*}z-\F+\hat{z}). 
\end{align}
Since $J(z)=\int_{\Om}f(z(x))dx$ is a convex functional on $L^{p^*}(\Om)$ and
 $z_m(t)$ converges weakly to
 $ z(t)$ in $L^{p^*}(\Om)$, we get that
  $\liminf_{m\to\infty}\int_{\Om}f(z_m(x,t))dx\geq \int_{\Om}f(z(x,t))dx$. 
Next,  for a.e. $x\in\Om$ we have that $z(x,\cdot)\in W^{1,p^*}(0,T;S^3\times\R^3)$. Let us fix $x\in\Om$. From (\ref{aprioriEstim5}) we conclude that the set $f(z(x,[0,T]))\subset \dom f$. Since $f\in C^1(\dom f)$ and $z(x,[0,T])$ is a compact subset of $\dom f$, we obtain that $f(z(\cdot)):[0,T]\to\R$ is an absolute continuous function and therefore $f(z(t))$ is almost everywhere strongly differentiable with $f(z(t))-f(z^0)=\int_0^t(\pa f(z(\tau),z_t(\tau))d\tau$  for a.e. $t\in (0,T)$. It follows from (\ref{aprioriEstim5}) that $\int_{\Om}(f(z(x,t))-f(z^0(x))dx<\infty$ and hence $\int_{\Om}\int_0^t(\pa f(z(x,\tau),z_t(x,\tau))d\tau dx<\infty$. Therefore we get\footnote{The existence of the integral 
  $\int_0^t \int_{\Om}(\pa f(z(x,\tau),z_t(x,\tau)) dxd\tau$ is an open problem.}
\begin{align}
\label{n_est9}
\liminf_{m\to\infty}\int_{\Om}(f(z_m(x,t))-f(z^0(x))dx\geq \int_{\Om}\int_0^t(\pa f(z(x,\tau),z_t(x,\tau))d\tau dx.
\end{align} 
Now, we estimate $\liminf_{m\to\infty}\li M \bar{z}_m,z_{mt}\re_{2,\Om_t}$ from below. 
To this end, we note first that  (\ref{n_est4'}) implies
 that $\liminf_{m\to\infty}\li M \bar{z}_m,z_{mt}\re_{2,\Om_t}=\liminf_{m\to\infty}\li M {z}_m,z_{mt}\re_{2,\Om_t}$. 
 The equations (\ref{equ1}) - (\ref{equ2}) yield that
\begin{align}
\label{n_est6'}
&\li Mz_m,z_{mt}\re_{2,\Om_t}=\Li -{\sigma_m\choose E_m}, z_{mt}\Re_{2,\Om_t}=\Li -{\sigma_m\choose E_m}, -\D^{-1} {\sigma_{mt}\choose E_{mt}}+{\eps(\nabla u_{mt})\choose D_{mt}}\Re_{2,\Om_t}
\non\\&=\Li {\sigma_m\choose E_m}, \D^{-1} {\sigma_{mt}\choose E_{mt}}\Re_{2,\Om_t}.
\end{align} 
We introduce the functions $w_m={\sigma_m\choose E_m}$. Then in terms of $w_m$  the last term can be rewritten as follows
\begin{align*}
\li \D^{-1}w_m,w_{mt}\re_{2,\Om_t}=\|(\D^{-1})^{\frac 12}w_m(t)\|_2^2-\|(\D^{-1})^{\frac 12}w(0)\|_2^2,
\end{align*}
where $w(0)=w_m(0)$ does not depend on $m$ since $z_m(0)=z^0$. The relation $\{w_m\}_m = \{Mz_m\}_m$ implies that $w_m\wto w$ in $W^{1,2}(0,T;L^2(\Om))$ and since $\D^{-1}$ is a positive definite symmetric operator, we obtain that $(\D^{-1})^\frac 12w_m(t)\wto (\D^{-1})^\frac 12w(t)$ in $L^2(\Om)$. Since $w\mapsto \|(\D^{-1})^{\frac 12}w\|_2^2$ is a proper, convex and l.s.c. functional on $L^2(\Om)$, using Lemma \ref{chainrule} we obtain
\begin{align}
\label{n_est7}
\liminf_{m\to\infty}\li \D^{-1}w_m,w_{mt}\re_{2,\Om_t}&= \liminf_{m\to\infty}(\|(\D^{-1})^{\frac 12}w_m(t)\|_2^2-\|(\D^{-1})^{\frac 12}w(0)\|_2^2)\non\\&\geq \|(\D^{-1})^{\frac 12}w(t)\|_2^2-\|(\D^{-1})^{\frac 12}w(0)\|_2^2=\li \D^{-1}w,w_t\re_{2,\Om_t}.
\end{align}
Let $(u,\phi,\sigma,D)$ be a solution of the problem (\ref{equ1}), (\ref{equ2}) with homogeneous boundary conditions corresponding to $z$. 
And let us assume first that $f$ satisfies the coercivity condition (\ref{grow_cond3'}). 
Then we have that $z_m\wto z$ in $L^p(\Om_T)$. Combining (\ref{n_est6'}), (\ref{n_est7}) we obtain
\begin{align}
\label{n_est7'}
&\liminf_{m\to\infty}\li Mz_m,z_{mt}\re_{2,\Om_t}\geq \li \D^{-1}w,w_t\re_{2,\Om_t}=\li w,\D^{-1}w_t\re_{2,\Om_t}\non\\&=\int_0^t\int_{\Om}\left({\sigma\choose E(\na\phi)},{\eps(\na u_t)\choose D_t}\right)dxds+\li Mz,z_t\re_{p,p^*,\Om_t}
\end{align}
Since $z\in L^{\infty}(0,T;L^p(\Om))$, then also ${\sigma\choose E(\na \phi)}\in L^{\infty}(0,T;L^p(\Om))$ and we obtain that the first term on the right side of (\ref{n_est7'}) $\li {\sigma\choose E(\na\phi)},{\eps(\nabla u_{t})\choose D_{t}}\re_{p,p^*,\Om_T}=0$.
Therefore
\begin{align}
\label{n_est7''}
&\liminf_{m\to\infty}\li Mz_m,z_{mt}\re_{2,\Om_t}\geq\li Mz,z_{t}\re_{p,p^*,\Om_t}.
\end{align}

Altogether we obtain
\begin{align}
\label{n_est12}
\hat{I}_{g^*}(z_{t})+\hat{I}_g(-M_{p}z-\F+\hat{z})\leq \li z_{t},-M_{p} z 
+{\hat{z}}\re_{p,p^*,\Om_t}-\int_{\Om}\int_0^t(z_t,\pa f(z)) dtdx.
\end{align}

Suppose now that $f(r,P)=f(P)$ and satisfies the coercivity condition (\ref{grow_cond3''}). The coercivity condition implies that $P\in L^{\infty}(0,T;L^2(\Om))$. Since $\pa f(z)={0\choose \pa f(P)}$ and the sequence $(\Sigma_m)_m$ is uniformly bounded in $L^p(\Om_T)$, we obtain that $\sigma\in L^p(\Om_T)$ and $\li\sigma,\eps_t \re_{\Om_t}=0$. Now since $E,D-P\in W^{1,2}(0,T;L^2(\Om))$ we integrate by parts the term $\li E,D_t-P_t\re_{\Om_t}$ with respect to $t$ to obtain
\begin{align*}
\li E,D_t-P_t\re_{\Om_t}=\li E,D-P\re_{\Om}|_0^t-\li E_t,D-P\re_{\Om_t}.
\end{align*}
Since $P\in L^2(\Om)$ and $D-P\in L^2(\Om)$ for every $t\in [0,T]$, then also $D\in L^2(\Om)$. Therefore we obtain for every $t\in [0,T]$ that $\li E,D\re_\Om =0$. And since $E_t=E(\nabla \phi_t)\in L^2(\Om)$, we obtain for a.e. $t\in [0,T]$ that also $\li E_t,D\re_\Om=0$. Thus we obtain that 
\begin{align}
\label{n_est13}
\li E,D_t-P_t\re_{\Om_t}=-\li E,P\re_{\Om}|_0^t+\li E_t,P\re_{\Om_t}.
\end{align}
 Now we show that the function $x\to \int_0^t(E,P_t)\in L^1(\Om)$. Indeed, we have that for every fixed $x\in\Om$ the function $(E,P)$ is absolutely continuous and 
\begin{align*}
(E,P)|_0^t=\int_0^t\frac{d}{dt}(E,P)ds=\int_0^t(E_t,P)+(E,P_t)ds. 
\end{align*}
It follows from the estimate (\ref{n_est13}) that $\int_\Om\int_0^t(E,P_t)dsdx$ exists. Therefore we get that the equality $\int_{\Om}\int_0^t(E,D_t-P_t)dsdx=-\int_{\Om}\int_0^t(E,P_t)dsdx$ is satisfied and with $m\to \infty$ the estimate (\ref{n_est3}) takes the form
\begin{align}
\label{n_est13''}
\hat{I}_{g^*}(z_{t})+\hat{I}_g(-M_{p^*}z-\F+\hat{z})\leq \li z_{t},{\hat{z}}\re_{p,p^*,\Om_t}-\int_{\Om}\int_0^t(z_t,M_{p^*}z+\pa f(z)) dtdx.
\end{align}

This completes the proof
of Theorem~\ref{existMain_rate_dep}.
\end{proof}

Now we prove the existence result for the case when $\L$ is a positive definite operator.

\paragraph{The proof of Theorem~\ref{existMain_rate_indep}.} 
\begin{proof} 
We start again from the inequality
\begin{align}
\label{n_est13'}
&\int_0^{t}\int_{\Om} g^*(z_{mt}(x,s))dsdx+\int_0^{t}\int_{\Om}g(\bar{\Sigma}_m(x,s))dsdx+ \li (M+\L) \bar{z}_m,z_{mt}\re_{2,\Om_t}\non\\&+\int_{\Om}f(z_m(x,t))dx-\int_{\Om}f(z^0(x))dx 
\leq\frac1{2m}\| z^0\|^2_2+\li z_{mt},\bar{\hat{z}}_m\re_{2,\Om_t}.
\end{align}
Using (\ref{aprioriEstim3''}) for the rate-dependent and (\ref{aprioriEst1}) for the rate-independent case we obtain based on the same arguments as in the (\ref{n_est4}) 
that up to a subsequence the sequence $\{\bar z_m-z_m\}_m$ converges strongly to $0$ in 
$L^2({\Om_T})$ and therefore the sequences $\{z_m\}_m$ and $\{\bar z_m\}_m$ converge weakly  to the same limit $z\in W^{1,2}(0,{T};L^{2}(\Om))$ in $L^{2}({\Om_T})$ and $\{z_{mt}\}_m$ converges weakly to $z_t\in L^2(\Om_T)$.

The functionals $\hat{I}_g:L^2(\Om_t)\to\bar{\R}$ and $\hat{I}_{g^*}:L^{2}(\Om_t)\to\bar{\R}$ defined by $\hat{I}_{g^*}(z)\equiv\int_0^{t}\int_{\Om}g^*(z(s,x))dsdx$ and $\hat{I}_g(z)\equiv\int_0^{t}\int_{\Om}g(z(s,x))dsdx$, respectively, are convex, proper and l.s.c. for both rate-dependent and rate-independent choices of the function $g$. Obviously, it holds
 $\hat{I}_{g^*}=(\hat{I}_g)^*$. Since $z_{mt}$ converges weakly to $z_t$ in $L^{2}(\Om_T)$ we obtain 
\begin{align}
\label{n_est14}
\liminf_{m\to\infty}\hat{I}_{g^*}(z_{mt})\geq \hat{I}_{g^*}(z_{t}). 
\end{align}

Due to (\ref{aprioriEstim2}) and (\ref{aprioriEst2}) there is a subsequence of $\{\bar{\Sigma}_m\}_m$, which converges weakly to $\bar{\Sigma}=-(M+\L)z-\F+\hat{z}$ in $L^2(\Om_T)$ with the function $\F\in L^2(\Om_T)$ such that $\F(t,x)=\int_{S^3\times\R^3}\pa f(\xi)\tau_{t,x}(d\xi)\ {\rm for}\ {\rm a.e.}\ (t,x)\in \Om_T$ with the Young measure $\tau\in L_w(\Om_T; \Meas(S^3\times\R^3))$ associated to the sequence $\{z_m\}_m$. Then we obtain
\begin{align}
\label{n_est15}
\liminf_{m\to\infty}\hat{I}_{g}(\bar{\Sigma}_m)\geq \hat{I}_{g}(-(M+\L)z-\F+\hat{z}). 
\end{align}

Since $z_m(t)$ converges weakly to $z(t)$ in $L^{2}(\Om)$, then $\liminf_{m\to\infty}\int_{\Om}f(z_m(x,t))dx\geq \int_{\Om}f(z(x,t))dx$. Let us fix $x\in\Om$. For every fixed $x\in\Om$ we have $z(x,\cdot)\in W^{1,2}(0,T;S^3\times\R^3)$. Since $f\in C^1(S^3\times\R^3,\R)$, we obtain that $f(z(\cdot)):[0,T]\to\R$ is an absolute continuous function and therefore $f(z(t))$ is for a.e. $t\in (0,T)$ strongly differentiable with $f(z(t))-f(z^0)=\int_0^t(\pa f(z(\tau),z_t(\tau))d\tau$. It follows from (\ref{aprioriEstim5}) and (\ref{aprioriEst5}) that $\int_{\Om}(f(z(x,t))-f(z^0(x))dx<\infty$ and hence $\int_{\Om}\int_0^t(\pa f(z(x,\tau),z_t(x,\tau))d\tau dx<\infty$. Therefore we have
\begin{align}
\label{n_est16}
\liminf_{m\to\infty}\int_{\Om}(f(z_m(x,t))-f(z^0(x))dx\geq \int_{\Om}\int_0^t(\pa f(z(x,\tau),z_t(x,\tau))d\tau dx.
\end{align} 
Due to the weak convergence of the sequences $(z_m)_m$ and $(\bar{z}_m)_m$ in the space $L^2(\Om_T)$ we obtain
\begin{align}
\label{n_est17}
&\liminf_{m\to\infty}\li (M+\L) \bar{z}_m,z_{mt}\re_{2,\Om_t}=\liminf_{m\to\infty}\li (M+\L) {z}_m,z_{mt}\re_{2,\Om_t}=\liminf_{m\to\infty}(\|(M+\L)^{\frac 12}z_m(t)\|_2^2\non\\&-\|(M+\L)^{\frac 12}z(0)\|_2^2)\geq \|(M+\L)^{\frac 12}z(t)\|_2^2-\|(M+\L)^{\frac 12}z(0)\|_2^2=\li (M+\L)z,z_t\re_{2,\Om_t}
\end{align}
Altogether we take the limit inferior on the left side and the limit on the right side of (\ref{n_est13'}) and  obtain with (\ref{n_est14})-(\ref{n_est17}) the following inequality
\begin{align}
\label{n_est10}
\hat{I}_{g^*}(z_{t})+\hat{I}_g(-(M+\L)z-\F+\hat{z})\leq \li z_{t},-(M+\L) z 
+{\hat{z}}\re_{2,\Om_t}-\int_{\Om}\int_0^t(z_t,\pa f(z)) dtdx.
\end{align}
 This completes the proof
of Theorem~\ref{existMain_rate_indep}.
\end{proof}

\begin{rem}
In the proof of Theorem \ref{existMain_rate_dep} and Theorem~\ref{existMain_rate_indep} 
the function $(r,P)$ is obtained as the limit of the weakly convergent sequence $(r_n,P_n)$ 
and the measure $\tau$ is the Young measure associated with this sequence $(r_n,P_n)$ 
(see Appendix \ref{young_measures} for the main properties of Young measures, if needed). 
As we have mentioned in Remark~\ref{meas_is_str}, the measure-valued solution $(u,\phi,r,P)$ is strong if  $\F=\na_{(r,P)}f(r,P)$ in (\ref{def_meas_sol}).  
This holds if the mapping $\na_{(r,P)}f$ is affine. However, this case has no practical applications
and therefore beyond our interests.
\end{rem}

\begin{appendix}
\section{Convex analysis}
\label{BasicsConAna}
In this section we briefly recall some basic facts about convex functions,
their subdifferentials and the surjectivity results for them.
 
Let $V$ be a reflexive Banach space with the norm $\|\cdot\|$, $V^*$ 
be its dual space with the norm $\|\cdot\|_*$. The
brackets $\left< \cdot ,\cdot \right>$ denote the duality pairing between
$V$ and $V^*$. By $V$ we shall always mean a reflexive Banach space
throughout this section.

For a function $\phi:V \to \bar{\R}$ the sets
\[\dom (\phi)=\{v\in V\mid \phi(v)<\infty\}, \ \
\epi(\phi)=\{(v,t)\in V\times \cR\mid \phi(v)\le t\}\]
are called the {\it effective domain} and the {\it epigraph} 
of $\phi$, respectively. One says that the function $\phi$ is {\it proper}
if $\dom(\phi)\not=\emptyset$ and $\phi(v)>-\infty$ for every $v\in V$.
The epigraph is a non-empty closed convex
set iff $\phi$ is a proper lower semi-continuous convex function or,
equivalently, iff $\phi$ is a proper weakly
 lower semi-continuous convex function 
(see \cite[Theorem 2.2.1]{Zalinescu02}, if needed).

 The Legendre-Fenchel
conjugate of a proper convex lower semi-continuous function $\phi : V \to 
\bar{\R}$ is the function $\phi^*$ defined for each 
$v^* \in V^*$ by
\[\phi^*(v^*)=\sup_{v \in V} \{ \left< v^*, v \right> - \phi(v)\}. \]
The Legendre-Fenchel conjugate $\phi^*$ is convex, lower semi-continuous
and proper on the dual space $ V^*$. Moreover, the 
{\it Young-Fenchel inequality} holds
\begin{eqnarray}\label{YoungFenchelIneq}
\forall v\in V,\ \forall v^*\in V^*:\ \ \phi^*(v^*)+\phi(v)\ge 
\left< v^*, v \right>,
\end{eqnarray}
and the inequality $\phi\le\psi$ implies $\psi^*\le\phi^*$ for any two
 proper convex lower semi-continuous functions $\psi,\phi: V \to 
\bar{\R}$ (see \cite[Theorem 2.3.1]{Zalinescu02}). 
Due to Proposition II.2.5 in \cite{Barb76} a proper convex 
lower semi-continuous function $\phi$ satisfies the following identity
\begin{eqnarray}\label{domainConvFunc}
\inter\dom(\phi)=\inter\dom(\partial\phi),
\end{eqnarray}
where $\partial\phi: V \to 2^{V^*}$ denotes the subdifferential 
of the function $\phi$. We note that the equality in (\ref{YoungFenchelIneq})
holds iff $v^*\in\partial\phi(v)$, i.e. together with (\ref{YoungFenchelIneq})
\begin{align}
\label{conv_ineq}
v^*\in\pa\phi(v)\Leftrightarrow \phi^*(v^*)+\phi(v)\leq \left< v^*, v \right>,\ \forall v\in V,\ \forall v^*\in V^*.
\end{align}
\begin{rem}\label{SubMax} We recall that the subdifferential of
 a lower semi-continuous proper and
convex function is maximal monotone\footnote{A monotone 
mapping $A:V \to 2^{V^*}$ is called {\rm maximal monotone} iff the
inequality
 \[\left< v^* - u^*, v - u \right> \ge 0  \ \ \ \ \forall \ u^*\in A(u)\]
implies $v^*\in A(v)$.}
 (see \cite[Theorem II.2.1]{Barb76}).
\end{rem}
The next surjectivity result on subdifferentials of convex functions is 
one of the key
tools in the proof of our main existence result.

\begin{theo}\label{SurjSubdifferential}
Let $A:=\partial\phi$, where $\phi:V\to\bar{\R}$ is
a proper convex lower semi-continuous function on $V$. Then the following
conditions are equivalent
\begin{eqnarray}\label{SurjConvFunc}
&&\lim_{\|v\|\to\infty}\frac{\phi(v)}{\|v\|}=\infty;\\
&& R(A)= V^* \; {\rm and}\;  A^{-1}\; {\rm is}\; {\rm bounded.}
\end{eqnarray}
\end{theo}
\begin{proof} See \cite[Theorem II.2.6]{Barb76}, for example.
\end{proof}
To state our next result, we recall that the relation
\[\partial\phi+\partial\psi=\partial(\phi+\psi)\]
holds for any two convex functions $\psi$ and $\phi$, if there exists
a point in $\dom(\phi)\cap\dom(\psi)$ where $\phi$ is continuous 
(see \cite[Proposition II.7.7]{Show97}). Then, since a proper convex 
lower semi-continuous function is continuous on the interior of its domain
(\cite[Proposition II.2.2]{Barb76}),
we get the following important result.
\begin{prop}\label{SumSubdiffConvFunc} Let $\phi$ be a proper convex 
lower semi-continuous function and $\psi$ be convex. Suppose that
\begin{eqnarray}\label{DomainSubdiffConvFunc}
\inter\dom(\phi)\cap\dom(\psi)\not=\emptyset.
\end{eqnarray}
Then\[\partial\phi+\partial\psi=\partial(\phi+\psi).\]
\end{prop}
We will use the following chain rule
\begin{lem}
\label{chainrule}
Let $\phi:H\to\R_\infty$ be proper, convex and l.s.c. on $H$. If $z,z_t\in L^2(0,T;H)$ and if there exists $g\in L^2(0,T;H)$ with $g\in\pa\phi(z)$ a.e. on $[0,T]$, then $\phi(z)$ is absolutely continuous on $[0,T]$ and
\begin{align*}
\frac{d}{dt}\phi(z(t))=\Li h(t),\frac{dz(t)}{dt}\Re_H
\end{align*} 
holds for every $h\in\pa\phi(u)$ a.e. on $[0,T]$.
\end{lem}
{\bf Convex integrands.}
For a proper convex lower semi-continuous function 
$\phi:{\R}^k\to\bar{\R}$ we define a functional $I_{\phi}$ on
$L^p(\Om,{\R}^k)$ by
\begin{align}
\label{conv_int} 
I_{\phi}(v)=\begin{cases}\int_\Om\phi(v(x))dx, & \phi(v)\in 
                                      L^1(\Om,{\R}^k)\\
                       +\infty, & {\rm otherwise}
  \end{cases},
 \end{align} 
where $\Om$ is a bounded domain in ${\R}^N$ with some
$N\in{\mathbb N}$. Due to
Proposition II.8.1 in \cite{Show97} the functional $I_{\phi}$
is proper, convex, lower semi-continuous, and $v^*\in\partial I_{\phi}(v)$ iff
\begin{eqnarray}
v^*\in L^{p^*}(\Om,{\R}^k), \ 
\ v\in L^p(\Om,{\R}^k)\ \ {\rm and} \ \ 
\ v^*(x)\in\partial{\phi}(v(x)), \ {\rm a.e.} \non
\end{eqnarray} 
Due to the result of Rockafellar in \cite[Theorem 2]{Rockafellar68}
the Legendre-Fenchel conjugate of $I_{\phi}$ is equal to $I_{\phi^*}$, i.e.
\[\big(I_{\phi}\big)^*=I_{\phi^*},\]
where $\phi^*$ is the Legendre-Fenchel conjugate of $\phi$.


\section{Young measures}
\label{young_measures}

Let $E\subset \R^m$ be a Lebesgue measurable set with $\mu(E)<\infty$. We denote $C_0(\R^d)=\overline{C_c(\R^d)}^{\|\cdot\|_\infty}$. Let $\Meas (\R^d)$ of signed Radon measures with bounded total variation. There is a one-to-one correspondence between the dual space of $C_0(\R^d)$ and the space $\Meas (\R^d)$, such that $\nu\in\Meas (\R^d)$ defines a linear continuous functional on $C_0(\R^d)$ in the following sense $\li \nu,f\re=\int_{\R^d}f(x)\nu(dx)$ for $f\in C_0(\R^d)$ and $\|\nu\|_{\Meas  (\R^d)}=\sup_{\|f\|_\infty\leq 1}|\li \nu,f\re|$.  One says that $\nu\in\Meas (\R^d)$ belongs to $ {\rm Prob}(\R^d)$ if $\nu$ is a probability measure. The mapping $\tau:E\to \Meas (\R^d)$ is said to belong to the space $L^\infty_w(E,\Meas (\R^d))$, if for all $f\in L^1(E;C_0(\R^d))$ the function $x\to\li \tau_x,f(x,\cdot)\re=\int_{\R^d}f(x,\la) d\tau_x(\la)$ is measurable and $\|\tau\|_{L^\infty_w(E;\Meas (\R^d))}\equiv\esssup_{x\in E}\|\tau_x\|_{\Meas (\R^d)}<\infty$.
\begin{df}
\label{young_def}
A Young measure $\tau:E\to\Meas (\R^d)$ is an element of the space $L^\infty_w(E,\Meas (\R^d))$, such that $\tau_x\in {\rm Prob}(\R^d)$ for $\mu-$a.e. $x\in E$.
\end{df}
\begin{theo}
\label{prohorov}
Let $\{u_n\}_n:E\to\R^d$ be a norm bounded sequence in $L^1(E;\R^d)$. Then there exists a subsequence $\{u_{n_k}\}_k$ of the sequence $\{u_n\}_n$ and a Young measure $\tau\in L^\infty_w(E,\Meas (\R^d))$ such that the sequence $\{\tau_{n_k}\}_k$ konverges to $\tau$ in $L^\infty_w(E,\Meas (\R^d))$.
\end{theo}
\begin{theo}
\label{young_exist}
Let $\{u_n\}_n:E\to\R^d$ be a sequence of measurable functions and $\{\tau_n\}_n$ be a sequence of Young measures associated to functions $\{u_n\}_n$ such that $\tau_n\to \tau$.
Let $\Phi:\R^d\to\R$ be continuous and suppose that $\{\Phi(u_n)\}_n$ is uniformly integrable.

Then $$\int_{E}\int_{\R^d}|\Phi(\xi)|\tau_x(d\xi)\mu(dx)<\infty$$
and $\Phi(u_n)$ converges  $(L^1,L^\infty)$ to $w$, where
$$w(x)=\int_{\R^d}\Phi(\xi)\tau_x(d\xi)\ {\rm for}\ \mu-{\rm a.e.}\ x\in E.$$
\end{theo}
\end{appendix}

\bibliographystyle{plain} 
{\footnotesize
\bibliography{literaturliste}
}
\section{Notations}
\label{Notations}

\begin{tabular}{l}
$(\cdot,\cdot),\ |\cdot|$: skalar product and norm in $\R^3$ or $S^3$\\
$\li\cdot,\cdot\re,\ \|\cdot\|_2$: skalar product and norm in $L^2(\Om)$ or $L^2(\Om)$\\
$\langle\cdot,\cdot\rangle_{p,p^*},\ \|\cdot\|_p$: bilinear form on $L^p(\Om)\times L^{p^*}(\Om)$ and norm in $L^p(\Om)$\\
$\li\cdot,\cdot\re_{k,p,\Om},\ \|\cdot\|_{k,p,\Om}$: skalar product and norm in $W^{k,p}(\Om)$\\
$[\cdot,\cdot]$: is a bilinear form on $W_0^{1,p}(\Om)\times W^{-1,p^*}(\Om)$.
\end{tabular}
\end{document}